\documentclass[draft]{amsart}
\usepackage[english]{babel}
\usepackage[normalem]{ulem}
\usepackage[all]{xy}
\usepackage{amssymb}
\numberwithin{figure}{section}
\numberwithin{equation}{section}
\usepackage{tikz}
\usetikzlibrary{shapes.geometric}
\usepackage{comment}
\usepackage{graphicx}
\usepackage{float}
\usepackage{xcolor}
\usepackage{xspace}
\usepackage{enumerate}
\usepackage{caption}
\usepackage{subcaption}
\usepackage{tkz-euclide}
\usepackage[utf8]{inputenc}
\usetkzobj{all}
\definecolor{ruta}{rgb}{0.309, 0.459, 0.208}
\definecolor{ruta1}{rgb}{0.205,0.301,0.201}
\definecolor{ruta2}{rgb}{0.409, 0.459, 0.208}
\definecolor{vino}{rgb}{0.256,0,0}
\definecolor{siva}{rgb}{0.205,0.201,0.201}
\definecolor{siva'}{rgb}{0.250,0.116,0.116}

\let\cal\mathcal

\def\Dscr{{\cal D}}

\def\Fscr{{\cal F}}
\def\Gscr{{\cal G}}

\def\Mscr{{\cal M}}

\def\Oscr{{\cal O}}

\def\Tscr{{\cal T}}

\def\Yscr{{\cal Y}}

\let\blb\mathbb

\def \ZZ{{\blb Z}}

\def \NN{{\blb N}}
\def \RR{{\blb R}}

\def\quot{/\!\!/}

\def\Qch{\operatorname{Qch}}
\def\coh{\mathop{\text{\upshape{coh}}}}

\def\rad{\operatorname {rad}}

\def\Spec{\operatorname {Spec}}

\def\Ext{\operatorname {Ext}}
\def\Hom{\operatorname {Hom}}

\def\End{\operatorname {End}}
\def\RHom{\operatorname {RHom}}

\def\End{\operatorname {End}}

\def\Pic{\operatorname {Pic}}

\def\gldim{\operatorname {gl\,dim}}

\def\r{\rightarrow}

\def\Vol{\operatorname {Vol}}
\def\rank{\operatorname {rank}}

\newcommand{\ol}{\overline}

\DeclareMathOperator{\Cl}{Cl}
\newcommand{\s}{\sigma}
\newcommand{\la}{\langle}
\newcommand{\ra}{\rangle}
\newcommand{\lf}{\lfloor}
\newcommand{\rf}{\rfloor}
\newcommand{\lc}{\lceil}
\newcommand{\rc}{\rceil}
\newcommand{\M}{{\sf{M}}}
\newcommand{\nP}{P'}

\def\convind{\text{convexly induced}\xspace} 
\def\convinde{\text{convexly induce}\xspace} 
\def\convindmod{\text{convexly induced}\xspace} 

\newcommand{\bc}{\mathtt{c}}
\newcommand{\bd}{\mathtt{d}}
\newcommand{\kl}{k}

\newtheorem{lemma}{Lemma}[section]
\newtheorem{proposition}[lemma]{Proposition}
\newtheorem{theorem}[lemma]{Theorem}
\newtheorem{corollary}[lemma]{Corollary}

\newtheorem{convention}[lemma]{Convention}

\theoremstyle{definition}

\newtheorem{definition}[lemma]{Definition}

{

\newtheorem{step}{Step}

}

\theoremstyle{remark}

\newtheorem{remark}[lemma]{Remark}

\newdimen\uboxsep \uboxsep=1ex
\def\uboxn#1{\vtop to 0pt{\hrule height 0pt depth 0pt\vskip\uboxsep
\hbox to 0pt{\hss #1\hss}\vss}}

\def\uboxs#1{\vbox to 0pt{\vss\hbox to 0pt{\hss #1\hss}
\vskip\uboxsep\hrule height 0pt depth 0pt}}

\def\Perf{\operatorname{Perf}}

\let\oldmarginpar\marginpar
\def\marginpar#1{ \oldmarginpar{\tiny \raggedright #1}}
\def\cm{}
\def\Sym{\operatorname{Sym}}

\title[Non-commutative crepant resolutions for some toric singularities II]{Non-commutative crepant resolutions for some toric singularities II}

\author[\v{S}pela
\v{S}penko and Michel Van den Bergh]{\v{S}pela \v{S}penko and Michel
  Van den Bergh} \thanks{The first author is a FWO $[$PEGASUS$]^2$
  Marie Sk\l odowska-Curie fellow at the Free University of Brussels
  (funded by the European Union Horizon 2020 research and innovation
  programme under the Marie Sk\l odowska-Curie grant agreement No
  665501 with the Research Foundation Flanders (FWO)). During part of
  this work she was also a postdoc with Sue Sierra at the University
  of Edinburgh.}  \address{Departement Wiskunde, Vrije Universiteit
  Brussel, Pleinlaan $2$, B-1050 Elsene}
\email[]{spela.spenko@vub.ac.be}
\address{Departement WNI, Universiteit Hasselt, Universitaire Campus \\
  B-3590 Diepenbeek} \email[]{michel.vandenbergh@uhasselt.be}
\thanks{The second author is a senior researcher at the Research
  Foundation Flanders (FWO).  While working on this project he was
  supported by the FWO grant G0D8616N: ``Hochschild cohomology and
  deformation theory of triangulated categories''.}
\keywords{Toric varieties,  tilting bundle, noncommutative resolution}
\subjclass[2010]{13A50, 14M25, 32S45}

\begin{document}
\begin{abstract}
Using the theory of dimer models Broomhead proved that every $3$-dimensional Gorenstein affine toric variety $\Spec R$ admits a {\emph toric} non-commutative crepant resolution (NCCR). 
We give an alternative proof of this result  by constructing a tilting bundle on a (stacky) crepant resolution of $\Spec R$ using standard toric methods. Our proof does not use dimer models.
\end{abstract}

\maketitle

\section{Introduction}
\label{sec:intro}

Throughout $k$ is an algebraically closed base field of characteristic 
zero.  Let $R$ be a normal
Gorenstein 
domain. A \emph{non-commutative crepant resolution} (NCCR)
\cite{DaoIyama,Leuschke,SVdB,VdB32,Wemyss1} of $R$ is an $R$-algebra
of finite global dimension of the form $\Lambda=\End_R(M)$ which in
addition is Cohen-Macaulay as $R$-module and where $M$ is a non-zero
finitely generated reflexive $R$-module. In \cite{SVdB} we studied NCCRs of
rings of invariants which are given by modules of covariants. More precisely, we assumed
 $R=\Sym(W)^G$ where~$G$ is a reductive
group and~$W$ is a representation of $G$ and $M=(U\otimes_k \Sym(W))^G$ for a  representation $U$ of~$G$. When $G$ is a product of a torus and a finite abelian group 
such invariant rings are  coordinate rings of affine toric 
varieties and in that case NCCRs given by modules of covariants are called ``toric'' NCCRs. See e.g.\ \cite{Bocklandt}.

\medskip

In loc.\ cit.\ we were able to construct toric NCCRs in many cases (e.g.\ when $W$ is self-dual). However the following 
beautiful result of Broomhead remained  outside the scope of our methods.
\begin{theorem}\cite[Theorem 8.6]{Broom}\label{Broom}
The coordinate ring of a $3$-dimensional Gorenstein affine toric variety admits a toric NCCR.   
\end{theorem}
Broomhead proves this result by exploiting the close relationship
 between $3$-dimensional Gorenstein affine toric varieties
and so-called ``dimer models'' (certain bipartite graphs embedded
in a real toric surface).  More precisely, it is well-known that the fan 
corresponding to a 3-dimensional Gorenstein affine toric variety is
a cone over a lattice polygon and 
an algorithm  by Gulotta
\cite{Gulotta} (or an alternative algorithm by Ishii and Ueda \cite{IshiiUeda})
 associates to every lattice polygon~$P$ a ``consistent'' dimer model whose lattice
polygon of ``perfect matchings'' coincides with~$P$. Any dimer model also possesses an
associated superpotential algebra and this yields the required toric NCCR.

\medskip

Instead of Broomhead's algebraic approach one may try to construct the NCCR as the endomorphism
algebra of a tilting bundle on a (stacky) crepant resolution of
singularities. In fact in \cite{SVdB4} we  achieved this in a relatively straightforward way by
modifying the main method of \cite{VdB32}. The resulting  NCCRs  are
 however generally non-toric and so we do not obtain a proof of Theorem \ref{Broom} in this way.
 
If we want to obtain toric NCCRs then the tilting bundle has to be \emph{split} ($=$ sum of line bundles). 
In \cite{IshiiUeda} Ishii and Ueda give indeed a proof of Theorem \ref{Broom}
 by constructing  a (tautological) split tilting bundle on a
toric crepant resolution  which is  the solution of a suitable
moduli problem related to dimer models (see \cite[Theorem 1.4]{IshiiUeda}, \cite[Theorem 6.4]{ishiiueda2}).

The purpose of this note is to give a proof of Theorem \ref{Broom}
which \emph{avoids} dimer models altogether and which uses traditional
toric methods instead.  In the interest of full disclosure we mention
that our proof still depends on the initial combinatorial input
provided by Gulotta and Ishii-Ueda.  A distinguishing feature of our
approach is that we focus on ``small'' toric resolutions, i.e. those
without exceptional divisors.  Except for the most simple cases such
resolutions cannot be schemes and must be Deligne-Mumford stacks. In
particular they cannot be realized as moduli spaces of vector bundles
(in contrast to the projective crepant resolutions used by Ishii and
Ueda) as they generally have non-connected stabilizers.

The version of our proof starting from Gulotta's approach is the least combinatorially demanding, so we discuss it first.
Let $\Spec R$ be a 3-dimensional Gorenstein affine toric variety\footnote{As in \cite{CoxLittleSchenck} we assume throughout that a normal toric affine variety is  of the form $\Spec(k[\sigma^{\vee}\cap M])$ for a rational strongly convex polyhedral cone in $N$ (see e.g. \cite[Remark 1.2.19(1), Theorem 1.3.5]{CoxLittleSchenck}).} 
and let $P$ be the lattice polygon corresponding to $R$ (see \S\ref{sec:prel}). Gulotta's ``algorithm'' \cite{Gulotta} 
first embeds~$P$ into a rectangle $P_0$ and then reconstructs $P$ from~$P_0$ by iteratively  removing 
{corner} triangles (see Figure \ref{fig:gulotta}).  
After appropriately subdividing the removed triangles
we obtain a triangulation 
of $P_0-P$ and we complete it to a triangulation of $P_0$ by  adding diagonals to~$P$. The affine toric variety $\Spec R_0$ corresponding to $P_0$ is a so-called ``generalized conifold'' and it has a standard toric NCCR (see Remark \ref{rmk:nccr0}). We show that the latter is in fact obtained from a {split}
tilting bundle  on the stacky crepant resolution~$\Yscr_0$ of   $\Spec R_0$ corresponding to the  
triangulation of $P_0$ (see Appendix \ref{stacks}). 
 The restriction to the open substack~$\Yscr$ of~$\Yscr_0$ corresponding to the triangulation of $P$ is then shown to be a tilting bundle on~$\Yscr$. 
Note that this is not a formality
as being tilting is not a local property.

Another version of our proof starts with the Ishii-Ueda approach  (see \S\ref{IshiiUeda}).
 In this case we embed~$P$ in a triangle $P_0$ and we reconstruct
 $P$ from $P_0$ by iteratively removing 
 {corner} vertices and taking convex
 hulls of the remaining lattice points. 
 {See Figure \ref{fig:iu}}. This yields again a natural triangulation of $P_0-P$ and we continue 
 {in a similar manner} as above. 

\medskip

Besides the above results we also discuss some side results in Appendices \ref{appA}, \ref{appC}
(Appendix \ref{stacks} is devoted to a technical result necessary for the paper).

In Proposition \ref{prop:Vol} we give a useful combinatorial criterion for recognizing NCCRs of
three-dimensional toric Gorenstein singularities.

In Appendix \ref{appC} we elaborate on the relationship between tilting bundles and NCCRs.
In \cite{IUconj} it was established by Ishii and Ueda that if $X$ is
an arbitrary \emph{projective} crepant resolution of $\Spec R$ then every NCCR of $\Spec R$ 
is the endomorphism ring of some tilting bundle on $X$.
As the proof depends on realizing $X$ as a GIT moduli space it appears to use 
the projectivity hypothesis in an essential way 
and  we do not know if the result is true otherwise.
We note however that there certainly exist particular instances of non-projective crepant resolutions
for which the result remains true.
In appendix \ref{appC} we discuss such an example using our combinatorial techniques.

\section{Acknowledgement}
The authors thank Seung-Jo Jung, 
Martin Kalck, Sasha Kuznetsov and Michael Wemyss for interesting discussions concerning the material in this paper.

\section{Preliminaries and notation}\label{sec:prel}
Let $N=\ZZ^n$, $M=\Hom(N,\ZZ)$ and let $\la \;,\;\ra$  the natural pairing between $M$ and~$N$.
Let $\{n_i\in N\mid 1\leq i\leq \kl\}$ be a set of vectors which generates $N_{\RR}$ such that
\[\s=[n_1,\dots,n_\kl]=\{r_1 n_1+\cdots+r_\kl n_\kl\in N_\RR\mid r_i\geq 0\}\]
is a rational strongly convex polyhedral cone in $N_\RR$. 
We do {not} assume that $\{n_1,\dots,n_\kl\}$ is a minimal set of generators of $\s$.  
Let 
$$\s^\vee=\{m\in M_\RR\mid \la m,n_i\ra\geq 0 \;\text{for all $1\leq i\leq \kl$}\}$$  be the dual cone of $\s$.  
We denote by  $R_\s=k[\s^\vee\cap M]$ the associated semigroup algebra, and set $X_\s=\Spec R_\s$.  

Let $\rho:M\hookrightarrow\ZZ^k$ be defined by $m\mapsto (\la m,n_1\ra,\dots,\la m,n_\kl\ra)$. We set $G=\Hom(\ZZ^\kl/\rho(M),k^*)\subseteq (k^*)^\kl$ so that $X(G)=\ZZ^\kl/\rho(M)$. 

Let $e_i$ be the $i$-th generator of $\ZZ^\kl$, and let $\beta_i$ be its image in $X(G)$. Then $M\cong \rho(M)=\{(a_i)\in \ZZ^\kl\mid \sum_i a_i \beta_i=0\}$, and 
$\sigma^\vee\cap M\cong \rho(\sigma^\vee\cap M)=\{(a_i)\in \NN^\kl\mid \sum_i a_i \beta_i=0\}$. 
Thus $R_\s\cong k[x_1,\dots,x_\kl]^G$ where $G$ acts on $x_i$ by  
$\beta_i$, and $X_\sigma=Y\quot G$ for $Y=k^\kl$ with the $G$-weights  $(-\beta_i)_{i=1}^\kl$. 

For $b\in \ZZ^\kl$ let 
\begin{equation}
\label{eq:Mbdef}
M_b=\{m\in M\mid \la m,n_i\ra\geq 
{-}b_i \;\text{for all $1\leq i\leq \kl$}\},\quad {\sf M}_b=k  M_b.
\end{equation}
Note that $R_\sigma={\sf M}_0=kM_0$ and 
$b-b'\in \rho(M)$ implies $M_b\cong M_{b'}$ as modules over $M_0$. 
Denote the image of $b$ in $X(G)$ by~$\chi_b$. To $m\in M_b$ we associate $(a_i)_i=(\langle m,n_i\rangle+b_i)_i\in \NN^\kl$.
With this identification we get $M_b\cong\rho(M_b)+b=\{(a_i)\in \NN^\kl\mid \sum_i a_i\beta_i=
{-} \chi_b\}$. 
Thus $\M_b$ is isomorphic to the module of covariants $M(\chi_{-b})=(k[x_1,\dots,x_\kl]\otimes \chi_{-b})^G$. 

\medskip

If $[n_1,\dots,n_k]$ is a minimal presentation of $\sigma$ then the action of $G$ on $Y$ is generic in the sense of  \cite[Def.\S 1.3.4]{SVdB} (see \cite[\S 11.6.1]{SVdB}).
In that case $\chi_{-b}\mapsto M(\chi_{-b})=\M_b$ provides isomorphism $X(G)=\Cl(R_\sigma)$ (see e.g. \cite[Section 5.1]{CoxLittleSchenck} or \cite[Lemma 4.1.3]{SVdB}). 
In particular $\M_b$ is reflexive and 
\begin{equation}\label{eq:refhom}
\Hom_{R_\sigma}(\M_b,\M_{b'})=\M_{b'-b}.
\end{equation}  

We will often consider the special case 
where the $(n_i)_i$ are the generators of the $1$-dimensional cones in a simplicial fan $\Sigma$ with $|\Sigma|=\sigma$.
In that case we will denote the corresponding  stacky fan\footnote{A stacky fan is a simplicial fan together with generators for the one dimensional cones.} 
$(\Sigma,(n_i)_{i=1}^\kl)$ by $\bold{\Sigma}$ (see \cite{BorisovHorja}).  We denote by $X_{\bf \Sigma}$ (resp. $X_\Sigma$) the corresponding smooth toric DM stack (resp. (possibly) singular toric variety). As $\Sigma$ is simplicial, $X_{\bf \Sigma}$ is  an orbifold and $X_\Sigma$ is its coarse moduli space.

We have that $X_{\bf \Sigma}$ (resp. $X_\Sigma$) equals $Y_\Sigma/ G$  (resp. $Y_\Sigma\quot G $) 
 for an open subvariety $Y_\Sigma\subset Y$ which consists of all points 
$z =(z_1,\dots,z_\kl)$ such that the set of $n_i$ for which $z_i=0$ is contained in a cone of $\Sigma$ (see \cite[Section 2]{BorisovHorja}). In the following
diagram we assemble the varieties/stacks we have defined and we give names to some of the canonical maps between them.
\begin{equation}
\label{eq:can}
\xymatrix{
X_{\bold{\Sigma}}\ar[d]_{\pi_s}&Y_{\Sigma}/G\ar[l]_{\mu_s}^{\cong}\ar[d]\ar@{^(->}[r]^{\theta_s}&Y/G\ar[d]\\
X_\Sigma \ar@/_2em/[rrr]_{\tau}&Y_\Sigma\quot G\ar[l]_{\cong}^{\mu}\ar[r]_{\theta}& Y\quot G\ar[r]^{\cong}&
X_\sigma
}
\end{equation}

Performing the construction $b\mapsto \M_b$ for each of the cones in $\Sigma$ we obtain a reflexive rank 1 $\Oscr_{X_{\Sigma}}$-module
which we denote by $\Mscr_{\Sigma,b}$. Recall the following result.
\begin{lemma} \label{lem:relations}
Put
\begin{equation*}
\label{def:linebundle}
\Mscr_{\bold{\Sigma},b}=\mu_{s,\ast}(\chi_{-b}\otimes \Oscr_{Y_\Sigma/G}).
\end{equation*}
Then $\Mscr_{\bold{\Sigma},b}$ is a line bundle on $X_{\bold{\Sigma}}$ and 
\[
\ZZ^k\mapsto \Pic(X_{\bold{\Sigma}}):b\mapsto \Mscr_{\bold{\Sigma},b}
\] is a  morphism of abelian groups.
We have
\begin{equation}
\label{def:reflexive}
\Oscr(D_b)=\Mscr_{\Sigma,b}=\pi_{s,\ast} \Mscr_{\bold{\Sigma},b}
\end{equation}
where $D_b=\sum_{i=1}^k b_i D_{n_i}$ is the equivariant Weil divisor associated to 
$b$ as in \cite[\S4.1]{CoxLittleSchenck}. Moreover
\begin{equation}
\label{def:module}
\Gamma(X_\Sigma,\Mscr_{\Sigma,b})=\M_b.
\end{equation}
\end{lemma}
\begin{proof} Since $\mu_s$ is an isomorphism the first two claims are obvious. The first equality in \eqref{def:reflexive} is \cite[Proposition 5.3.7]{CoxLittleSchenck}.
The last equality in \eqref{def:reflexive} is local on $X_\Sigma$ and can be verified on the open covering given by the maximal cones using Proposition \ref{prop:ocovst}. 
Finally \eqref{def:module} follows from \cite[Proposition 4.3.2]{CoxLittleSchenck}.
\end{proof}

Furthermore, for $m\in M$ we have by \cite[Theorem 9.1.3]{CoxLittleSchenck}
  \begin{equation}
\label{eq:cohomMb}
H^p(X_{\Sigma},\Oscr(D_{{b}}))_m=\tilde H^{p-1}(V_{D_{{b}},m},k),
\end{equation}
  where 
\begin{align}\label{vdbm}
V_{D_{{b}},m}&=\bigcup_{\s\in \Sigma} {\rm conv}\{ v\in
  \s(1)\mid \la m,v\ra<-b_v\}\\\nonumber
&=\bigcup_{\s\in \Sigma} {\rm conv}\{v\in
  \s(1)\mid s^b_v(m)=-\},
\end{align} 
and where we denote 
\begin{equation}
\label{eq:signdef0}
 s^b_{v}(m)=\text{sign of }(\la m,v\ra +b_v). 
\end{equation}
\begin{convention}{}
\label{conv}
Except for \S\ref{subsec:convind},\ref{sec:triangle} we assume that there is $m\in M$ such that $\la m,n_i\ra=1$ for all $1\leq i\leq k$.
In particular
$X_\s$ is Gorenstein  
(see e.g. \cite[Proposition 11.4.12]{CoxLittleSchenck}).  Changing bases we may assume $m=(0,\ldots,0,1)$
and hence $\s$ is a cone over a convex lattice polyhedron $P\subseteq \RR^{n-1}\times\{1\}\subset \RR^n=M_{\RR}$, and we write $\sigma=\sigma_P$. 
We abbreviate $R_{\s_P}, X_{\s_P}$ as $R_P$, $X_P$, respectively. 
\end{convention}

\subsection{Tilting bundles and NCCRs}
Recall that a tilting bundle on an algebraic stack $\Yscr$ is a vector bundle $\Tscr$ such that the following conditions
hold:
\begin{enumerate}
\item[(T1)]
$\Tscr$ generates $D_{\Qch}(\Yscr)$;
\item[(T2)]
$\Ext^i_{\Yscr}(\Tscr,\Tscr)=0$ for $i>0$.
\end{enumerate}
We then have $D_{\Qch}(\Yscr)\cong D(\Lambda)$ where $\Lambda=\End_{\Yscr}(\Tscr)$. 
If $\Yscr$ is a noetherian Deligne Mumford stack then 
we write $\Dscr(\Yscr)$  for $D^b_{\coh}(\Yscr)$.  
Similarly if $\Lambda$ is a noetherian ring then we write 
$\Dscr(\Lambda)=D^b_f(\Lambda)$.
If $\Yscr$ is a smooth separated DM-stack then $D_{\Qch}(\Yscr)$ is compactly generated by $\Dscr(\Yscr)$ (see \cite[Theorem B]{HallRydh})
and hence if $\Tscr$ is a tilting bundle on $\Yscr$ then $\Tscr$ ``classically
generates'' $\Dscr(\Yscr)$ (see \cite{BondalVdB}) and moreover if $\Lambda=\End_{\Yscr}(\Tscr)$ is noetherian then
$\Dscr(\Yscr)\cong \Dscr(\Lambda)$. 
Recall the following
\begin{proposition} \label{sec:standard}
Let $P$ be as in Convention \ref{conv}. Choose a triangulation of $P$ without extra vertices\footnote{The existence of such triangulations follows from the theory of secondary
fans \cite{GKZ,GKZbook}. See \cite[Proposition 15.2.9]{CoxLittleSchenck}.} 
and let $\Sigma$ be the corresponding fan. 
Let $\Tscr$ be a tilting bundle on $X_{\bold{\Sigma}}$. Then $\Lambda=\End_{X_{\bold{\Sigma}}}(\Tscr)$ is an NCCR for $R_P$ corresponding
to $M=\Gamma(X_{\bold{\Sigma}},\Tscr)$.

If $\Tscr$ is toric 
of the form
$\Tscr=\bigoplus_{b\in S} \Mscr_{\bold{\Sigma},b}$ then $\Lambda=\End_{R_P}(\bigoplus_{b\in S} \M_b)$.
\end{proposition}
\begin{proof} 
We proceed similarly as in the proof of \cite[Corollary 2.4]{SVdB4}. 
Since $X_{\bf\Sigma}$ is smooth and $\Dscr(X_{\bf\Sigma})\cong \Dscr(\Lambda)$, $\gldim \Lambda<\infty$. 
It follows from Proposition \ref{prop:ocovst} (c.f. proof of Lemma \ref{rmk:crepant}) that $\pi_s$ is identity in codimension $1$. Hence $(\pi_s)_*$ defines a (monoidal) equivalence between categories of reflexive sheaves on $X_{\bf\Sigma}$ and $X_\Sigma$. As $\tau$ does not contract a divisor by the hypothesis on the triangulation, 
 $\tau_*$ defines a (monoidal) equivalence between categories of reflexive sheaves on $X_\Sigma$ and $X_P$. Thus, $\End_{X_\Sigma}(\Tscr)\cong \End_{R_P}(\Gamma(X_{\bf\Sigma},\Tscr))$ as $\Gamma(X_P,(\tau\pi_s)_*\Tscr)=\Gamma(X_{\bf\Sigma},\Tscr)$. 

To show that $\Lambda$ is Cohen-Macaulay one can proceed as in \cite[Corollary 2.8]{SVdB4} thanks to Lemma \ref{rmk:crepant}. The last claim follows from \eqref{def:module}.
\end{proof}

\section{Strategy}
As alluded to in the introduction our main steps towards the construction of an NCCR of $R_P$ (for $2$-dimensional $P$) are the following:
\begin{enumerate}
\item\label{s1}{}
Embed $P$ in a rectangle (resp. lattice triangle) $P_0$, and choose a triangulation of $P_0$ which contains a triangulation of $P$ without extra vertices.
To obtain these triangulations we use
 Gulotta's (see \S\ref{Gulotta})
(resp. the Ishii-Ueda (see \S\ref{IshiiUeda})) inductive procedures. 

\noindent
Let $\Sigma_0$, $\Sigma$ be the fans corresponding to the triangulations of $P_0$ and $P$.
\item\label{s2}
Construct a split tilting bundle $\Tscr_0$ on the stacky crepant resolution $X_{\bold{\Sigma}_0}$ of $R_{P_0}$.

\item\label{s3}
Restrict $\Tscr_0$ to $X_{\bold{\Sigma}}$ to obtain a tilting bundle $\Tscr$ on $X_{\bold{\Sigma}}$ (this is not a formal step as the property of being tilting is not local).

\item\label{s4}
Now $\Lambda=\End_{X_{\bold{\Sigma}}}(\Tscr)$ yields an NCCR  of $R_P$ by Proposition \ref{sec:standard}.
\end{enumerate}
The vanishing
properties (T2) for $\Tscr_0$, $\Tscr$ in \eqref{s2}\eqref{s3} are proved using standard toric geometry (see Lemmas \ref{lem:pretilting}, \ref{lem:pretiltingIU}), while the generation property (T1) for $\Tscr$ in \eqref{s3}
follows formally from the corresponding property of $\Tscr_0$ in \eqref{s2} (generation is compatible with restriction).
The verification of the  latter constitutes the heart of our proof.
 We proceed as follows:

\begin{enumerate}[(a)]
\item\label{s21} We start with a standard NCCR $\Lambda_0$ of
  $R_{P_0}$ given by $\oplus_{b\in S_0}\mathsf{M}_{b}$, $S_0=\{(b_v)_{v\in
    V_0'}\}$, where $V'_0$ is the set of the vertices of $P_0$ ($V'_0$ yields the \emph{minimal} presentation of
  $\sigma_{P_0}$). Clearly, depending on whether $P_0$ is a rectangle or a triangle, we have $|V'_0|=4$ or $3$.  

\item\label{s22} 
Let $V_0$ be the set of vertices in the triangulation in \eqref{s1}
($V_0$ gives a \emph{non-minimal} presentation of $\sigma_{P_0}$). We
develop an inductive procedure called \emph{compatible convex
  induction} (see \S\ref{subsec:convind}) which extends the collection
$S_0=\{(b_v)_{v\in V'_0}\}$ from \eqref{s21} to a collection
${S}=\{(\tilde{b}_v)_{v\in V_0}\}$ (thus $S_0$ and ${S}$ are in bijection  via $\tilde{(-)}$) such that for $b,b'\in S_0$:
\begin{equation}
\label{eq:compatprops}
M_b=M_{\tilde{b}}\text{ and }M_{b-b'}=M_{\tilde{b}-\tilde{b}'}.
\end{equation}
The construction of $\tilde{b}$ from $b$ is is done inductively and matching the inductive
 construction of the triangulation of $P_0$ (see \eqref{s1}).
\item\label{s2tilt}
We define $\Tscr_0:=\oplus_{b\in {S}}\Mscr_{\bold{\Sigma}_0,b}$ 
and we 
verify using the standard results on the cohomology of rank one reflexive 
sheaves on 3-dimensional toric varieties 
(see Lemmas \ref{lem:pretilting}, \ref{lem:pretiltingIU}) 
that $\Tscr_0$ satisfies (T2). This verification is done inductively.

\item\label{s2Lambda0}
We use \eqref{eq:compatprops} to deduce that $\Lambda_0\cong \End_{X_{\bold{\Sigma}_0}}(\Tscr_0)$. 

\item\label{semiorth}
We show that $\Dscr(X_{\bold{\Sigma_0}})$ has no non-trivial semi-orthogonal decompositions using the fact that it has trivial relative Serre functor over $X_P$ (see Lemma \ref{relSerreaff}, Corollary \ref{cor:sod}).

\item\label{s2sod}
Since $\Perf(\Lambda_0)\cong \Perf(\End_{X_{\bold{\Sigma}_0}}(\Tscr_0))$ 
 is fully faithfully embedded in $\Dscr(X_{\bold{\Sigma_0}})$ by (T2) (see \eqref{s2tilt}), and since $\gldim \Lambda_0<\infty$ by \eqref{s21} we have that $\Perf(\Lambda_0)=\Dscr(\Lambda_0)$ and moreover the embedding is admissible by \cite[Lemma 1.1.1]{PolVdB}. So it is a factor in semi-orthogonal decomposition of  $\Dscr(X_{\bold{\Sigma_0}})$.
We now use \eqref{semiorth}  to conclude that $\Dscr(\Lambda_0)\cong \Dscr(X_{\bold{\Sigma}_0})$  and hence $\Tscr_0$ satisfies (T1).
\end{enumerate} 

\medskip

Our constructions were inspired by Bocklandt's result \cite[Corollary 4.7]{Bocklandt} showing that Cohen-Macaulay modules of covariants are ``preserved under projection''.
In fact using Bocklandt's result we may prove that $\Lambda$ is Cohen-Macaulay, without (re)invoking toric geometry (see Appendix \ref{appA} for more details).  However we have been unable
to prove directly that $\Lambda$ is an NCCR. 

\section{Convex induction}\label{subsec:convind}
For $t=(t_1,\dots,t_k)$, $0\leq t_i\leq 1$, $\sum_i t_i=1$, we set 
\begin{equation}
\label{eq:basic}
b_{t,-}=(b_1,\dots,b_k,\left\lfloor \sum_i t_i b_{i}\right\rfloor),\quad
b_{t,+}=(b_1,\dots,b_k,\left\lceil \sum_i t_ib_{i}\right\rceil),
\end{equation}
where $\lfloor x\rfloor$, $\lceil x\rceil$ denote the largest (resp. the smallest)  integer not greater (resp. smaller) than $x$. 

\begin{lemma}\label{newpresenetation}
If $\s$ is presented as $[n_1,\dots,n_k,n_{k+1}]$, 
where $n_{k+1}=\sum_i t_i n_i\in N$ for some $t=(t_1,\dots,t_k)$,  $0\leq t_i\leq 1$, $\sum_i t_i=1$, then 
$M_b=M_{b_{t,-}}=M_{b_{t,+}}$. 
\end{lemma}
\begin{proof}
  The inclusions $M_{b_{t,-}},M_{b_{t,+}}\subseteq M_b$ are trivial,
  while for $m\in M_b$ we have $\la m,n_{i}\ra\geq -b_{i}$ for all $i$,
  and thus
  $\la m,n_{k+1}\ra \geq -\sum_i t_i b_{i}\geq -\lceil \sum_i t_i
  b_{i}\rceil$,
  so $m\in M_{b_{t,+}}$, and as $\la m,n_{k+1}\ra \in \ZZ$ we also
  have $\la m,n_{k+1}\ra \geq -\lfloor \sum_i t_i b_i\rfloor$, implying
  $m\in M_{b_{t,-}}$.
\end{proof} 

If  $a\in \RR$ then by convention its ``sign'' will be $+$ if $a\geq 0$ and $-$ otherwise.
For $b=(b_v)_{v\in S}\in \ZZ^l$, with $S\subset N$, and $m\in M$
recall (from \eqref{eq:signdef0})
\begin{equation}
\label{eq:signdef}
 s^b_{v}(m)=\text{sign of }(\la m,v\ra +b_v). 
\end{equation}
If we use indexed vectors $(v_i)_i$ then we also write $s^b_i(m)$ for $s^b_{v_i}(m)$. 
\begin{remark}
The signs $s^b_v(m)$ appear often in toric geometry. See e.g.\ \eqref{eq:cohomMb} and Proposition \ref{lemmaB} below.
\end{remark}
Let notation be as in Lemma \ref{newpresenetation}. Applying \eqref{eq:signdef} with $(v_1,\ldots,v_k,v_{k+1})=(n_1,\ldots,n_k,n_{k+1})$
 we have $s^b_i(m)=s^{b_{t,+}}_i(m)=s^{b_{t,-}}_i(m)$ for $1\leq i\leq k$ since  $(b_{t,\pm})_i=b_i$ for $1\leq i\leq k$.
We will need the following simple lemma.
\begin{lemma}\label{lem:sign}
We have $s^{b_{t,-}}_{k+1}(m),s^{b_{t,+}}_{k+1}(m)\in \{s^b_i(m)\mid 1\leq i\leq k, t_i\neq 0\}$. 
\end{lemma}
\begin{proof} We may assume that all signs $s^b_i(m)$ for  $1\leq i\leq k$, $t_i\neq 0$ are the same, otherwise there is nothing to prove. 
Assume that they are all $-$; i.e. $\langle m,n_i\rangle+b_i\le -1$ for $1\leq i\leq k$, $t_i\neq 0$. Then $\langle m,n_{k+1}\rangle +\sum_i t_i b_i\le-1$.
Applying $\lfloor?\rfloor$, $\lceil?\rceil$ yields the desired conclusion. The case where all signs are $+$ is similar.
\end{proof}
\begin{definition} Let $k,l\in \NN$. An \emph{induction datum} $(t^{(i)}_j)_{ij}\in \RR$ is a collection of vectors $
t^{(i)}=(t^{(i)}_j)_j\in \RR^{i-1}$ for $k<i\le k+l$,
such that
$
0\leq t^{(i)}_j\leq 1,\quad \sum_{j=1}^{i-1} t^{(i)}_j=1
$.
\end{definition}
\begin{definition}
\label{def:convex1}
Let $[n_1,\dots,n_k]\in N^k$.  
We say  that 
$[n_1,\ldots,n_{k+l}]\in N^{k+l}$ 
is \emph{convexly induced} from  $[n_1,\dots,n_k]\in N^k$ with induction datum $(t_j^{(i)})_{i,j}$ if 
for $i>k$
\[
n_{i}=
\sum_{j=1}^{i-1}t^{(i)}_j n_j. 
\]
\end{definition}
\begin{definition}
Let $\mathfrak{s}\in \{\pm\}^{l}$ be a \emph{sign sequence}.
We say that an integer vector $\tilde{b}\in \ZZ^{k+l}$ is {\em
  \convindmod} from $b\in \ZZ^k$ with sign sequence $\mathfrak{s}$
and induction datum $(t_j^{(i)})_{i,j}$
if 
$\tilde{b}$ is  obtained from $b$ 
via $\tilde{b}_{i}=(\tilde{b}_{\leq i-1})_{t^{(i)},\mathfrak{s}_i}$ (see \eqref{eq:basic}) for $i>k$; i.e.
\[
\tilde{b}_i=
\begin{cases}
b_i &\text{if $1\leq i\leq k$}, \\
\lfloor\sum_{j=1}^{i-1}t^{(i)}_j \tilde{b}_{j}\rfloor &\text{if $i>k$ and $\mathfrak{s}_i=-$},\\
\lceil \sum_{j=1}^{i-1}t^{(i)}_j \tilde{b}_{j}\rceil&\text{if $i>k$ and $\mathfrak{s}_i=+$}.
\end{cases}
\]
  We say that $\tilde{b}$ is convexly induced
from $b$ if it is convexly induced for \emph{some} sign sequence (and a given induction datum).

Let $K\in \NN$. We say that
$\{\tilde{b}^{j}\mid 1\leq j\leq K\}\subset\ZZ^{k+l}$ is {\em
  compatibly \convindmod} from
$\{b^{j}\mid 1\leq j\leq K\}\subset \ZZ^k$ if for each $j$, $\tilde{b}^{j}$ is
convexly induced from $b^{j}$ for a fixed sign sequence and a fixed induction datum (i.e.\ they
are both independent of $j$).
\end{definition}
Below we will assume the induction datum is specified once and for all and we will usually not mention it afterwards.
\begin{corollary}\label{convind}
If $\tilde{b}$ is {\convindmod}  from $b$ and $I$ is a subsequence of $1,\ldots,k+l$ containing $1,\dots,k$
then $M_{b}=M_{b'}$ 
for
$b'=(\tilde{b}_i)_{i\in I}$.
\end{corollary}
\begin{proof}
Repeated application of  Lemma  \ref{newpresenetation} yields $M_{\tilde{b}}=M_{b}$.  This implies $M_b=M_{\tilde{b}}\subset M_{b'}\subset M_b$ which gives the desired conclusion.
\end{proof}
\begin{lemma}\label{diff}
If $\{\tilde{b}_1,\tilde{b}_2\}$ is compatibly \convind from $\{b,b'\}$ 
then $\tilde{b}_1-\tilde{b}_2$ is \convind from $b-b'$ (for the same induction datum and a possibly different sign sequence). In consequence, $M_{b'_1-b'_2}=M_{b_1-b_2}$ with $b'_1,b'_2$ like $b'$ in Corollary \ref{convind}. 
\end{lemma}
\begin{proof}
  By definition $\tilde{b}_1$, $\tilde{b}_2$ are convexly induced from
  $b_1,b_2$, respectively, with the same sign sequence.  Let $\ast$ be
  a sign. We do not necessarily have
  $(\tilde{b}_{1,\leq i-1})_{t^{(i)},\ast}-(\tilde{b}_{2,\leq
    i-1})_{t^{(i)},\ast}=((\tilde{b}_1-\tilde{b}_2)_{\leq
    i-1})_{t^{(i)},\ast}$.
  However, as\footnote{This innocent looking formula is in fact the
    key ingredient of our proofs!}
  $\lf x\rf-\lf y\rf,\lceil x\rceil-\lceil y\rceil\in\{\lf x-y\rf, \lc
  x-y\rc\}$
  for any $x,y\in \RR$, we still have
  $(\tilde{b}_{1,\leq i-1})_{t^{(i)},\ast}-(\tilde{b}_{2,\leq
    i-1})_{t^{(i)},\ast}=(\tilde{b}_{1,\leq i-1}-\tilde{b}_{2,\leq
    i-1})_{t^{(i)},\ast'}$
  for a possibly different sign $\ast'$. This is sufficient by
  Corollary \ref{convind}.
\end{proof}
\subsection{Interval convex induction}
A simple version of convex induction consists of lattice points on an interval with the induction datum only referring to the nearest neighbours. 
\begin{definition}
Let $n_0,n_{r+1}\in N$ be end points of an interval $I$.
 Let
$\{n_1,\ldots,n_{r}\}\in N\cap  I$  be such that $n_0,n_1,\ldots,n_r,n_{r+1}$ are consecutive distinct points in $I$.
Assume furthermore we are given another ordering of these points 
\begin{equation}
\label{eq:ordering}
n'_1,\ldots,n'_{r+2}=n_{0},n_{r+1},n_{i_{1}}\ldots,n_{i_{r}}.
\end{equation}
Then for any $3\le j\le r+2$ there exist unique $j',j''<j$ such that $]n'_{j'},n'_{j''}[\cap\allowbreak \{n'_1,\ldots,n'_j\}=\{n'_j\}$.
In particular have
\begin{equation}
\label{eq:ici}
n'_j=(1-t)n'_{j'}+tn'_{j''}
\end{equation}
for some $t\in ]0,1[$. The \emph{interval induction datum} $\mathfrak{t}$ associated to the ordering \eqref{eq:ordering} is the induction datum obtained from \eqref{eq:ici}.
%
%
%

We say that $(b_{n_0},b_{n_{r+1}},b_{n_1}\ldots,b_{n_r})\in \ZZ^{r+2}$ is obtained by \emph{interval convex induction}
from $(b_{n_0},b_{n_{r+1}})\in \ZZ^2$ if it is obtained by convex induction from $(b_{n_0},b_{n_{r+1}})$ for an interval induction datum  (associated to some ordering of  $(n_i)_{i=0,\dots,r+1}$ as in \eqref{eq:ordering}) 
and an arbitrary sign sequence.
\end{definition}
\begin{lemma} \label{lem:intervalconvexinduction}
Let $(n_i)_i$ be as above.
Assume $(b_{n_0},b_{n_{r+1}},b_{n_1}\ldots,b_{n_r})\in \ZZ^{r+2}$ is obtained by interval convex induction 
from $(b_{n_0},b_{n_{r+1}})\in \ZZ^2$.
Then for all $m\in M$ the signs $s_{n_j}^b(m)$, $j=0,\ldots,r+1$, follow the
pattern $+\cdots+-\cdots-$ (possibly reflected and perhaps with no $+$ or $-$ present). 
\end{lemma}
\begin{proof} This follows by repeated application of Lemma \ref{lem:sign}.
\end{proof}

\section{Triangle convex induction}
\label{sec:triangle}
In this section we analyze a type of the convex induction that we will use for constructing a tilting bundle on a stacky resolution via the Ishii-Ueda approach. It is similar in spirit, but considerably more involved, than  interval convex induction and it will be  used only in \S\ref{IshiiUeda}. 

\begin{definition}
Let $v_{-1},v_0,v_{r+1}\in N$ be vertices of a solid triangle $T$ such that $]v_{-1},v_0[$, $]v_{-1},v_{r+1}[$ contain no lattice points and let $C$ be the convex hull of $N\cap(T\setminus \{v_{-1}\})$.
Put $\{v_1,\dots,v_r\}= N\cap \partial C\cap \operatorname{int} T$  such that $v_0,v_1,\dots,v_{r},v_{r+1}$ are consecutive boundary points of $C$.

 For $j=1,\dots,r$, $v_j$ is in the convex hull of $v_{-1},v_{j-1},v_{r+1}$. In other words there exist unique
$p_j,q_j,r_j\in [0,1]$ such that $p_j+q_j+r_j=1$ and
\begin{equation}
\label{eq:inductiondatum0}
v_j=p_j v_{-1}+q_jv_{j-1}+r_jv_{r+1}.
\end{equation} 
The \emph{triangle induction datum} $\mathfrak{t}$ is the induction  datum obtained from \eqref{eq:inductiondatum0}.

We say that $(b_{-1},b_0,b_{r+1},b_1,\dots,b_r)\in \ZZ^{r+3}$ is obtained by \emph{triangle convex induction} from $(b_{-1},b_0,b_{r+1})$ if it is obtained by convex induction from $(b_{-1},b_0,b_{r+1})$ for 
$\mathfrak{t}$ and a sign sequence corresponding entirely of $-$'s; i.e., for $1\leq j\leq r$
\begin{equation}\label{eq:defbj0}
b_{j}=b_{v_j}=\lf p_jb_{-1}+q_jb_{j-1}+r_jb_{r+1}\rf.
\end{equation} 
\end{definition}
The following is our main result concerning triangle convex induction. It will play a similar role as Lemma \ref{lem:intervalconvexinduction}.
\begin{proposition}[see \S\ref{sec:combred}] \label{lem:sup:pretiltingIU}
Let  $m\in M$ and let $b,b'$ be obtained by triangle convex induction from $(b_{-1},b_{0},b_{r+1})$, $(b'_{-1},b'_{0},b'_{r+1})$.
Put $c=b'-b$. 
Then the possible sign patterns for $s^c_{v_{-1}}(m),s^c_{v_0}(m),\ldots,s^c_{v_{r+1}}(m)$ are
\[
-\underbrace{+\cdots+}_p\underbrace{-\cdots-}_q\underbrace{+\cdots+}_r \text{ or }
+\underbrace{-\cdots-}_p\underbrace{+\cdots+}_q\underbrace{-\cdots-}_r 
\]
where $p,q$ and $r$ may be equal to zero.
\end{proposition}

\subsection{An explicit expression for the triangle induction datum}
\subsubsection{Generalities}
\label{sec:recursion}
Assume $(a_l)_{l\in \ZZ}\in \RR$. Below we will consider solutions $(x_l)_{l}\in \RR$ to the second order recursion relation 
\begin{equation}
\label{eq:recursion}
x_{l+1}=a_{l}x_{l}-x_{l-1}.
\end{equation}
The following result follows easily from \eqref{eq:recursion}.
\begin{lemma} \label{lem:specialized} Assume $(x_l)_l,(y_l)_l\in \RR$ are solutions to \eqref{eq:recursion}.  Then
\[
\left|
\begin{matrix}
x_{l}&x_{l+1}\\
y_{l}&y_{l+1}
\end{matrix}
\right|
\]
is independent of $l$.
\end{lemma}
We define $(q_{s,t})_{s,t\in \ZZ}$ in such 
a way that  for all $s$, $(q_{s,t})_t$ is a solution to \eqref{eq:recursion} 
 with initial conditions $q_{s,s}=0$, $q_{s,s+1}=1$. Note that $q_{s,s-1}=-1$, $q_{s,s+2}=a_{s+1}$. 
The following equation shows that $(q_{s,t})_s$ is also a solution to \eqref{eq:recursion} and hence the indices $s,t$ of $q_{s,t}$ play symmetric roles:
\begin{equation}\label{p012}
q_{s+1,t}=a_{s}q_{s,t}-q_{s-1,t}.
\end{equation}
Indeed: both sides of \eqref{p012} are solutions to \eqref{eq:recursion} and they are equal for $t=s,s+1$ (using
 $q_{s-1,s+1}=a_{s}$).
For $(x_t)_t$ a solution to \eqref{eq:recursion} we get
\begin{equation}
\label{eq:recursion2}
x_t=q_{s,t}x_{s+1}-q_{s+1,t}x_s
\end{equation}
by a similar argument. Both sides of this equation are solutions to \eqref{eq:recursion} and moreover they agree for $t=s,s+1$. Applying
\eqref{eq:recursion2} with $(x_t)_t=(q_{s,t})_t$ (and the index $s$ replaced by $k$) we get
\begin{equation}
\label{eq:recursion3}
q_{s,t}=q_{s,k+1}q_{k,t}-q_{s,k}q_{k+1,t}.
\end{equation}
\subsubsection{Continued fractions}
Put $u_i:=v_i-v_{-1}$. After choosing a suitable basis for the plane spanned by $u_0$, $u_{r+1}$ and performing a translation we may assume $N=\ZZ^2$ and
\begin{equation}
\label{eq:us}
v_{-1}=(0,0), \quad v_0=(0,1), \quad v_{r+1}=(n,-q)
\end{equation}
where $n>0$, $0\le q<n$ and $\gcd(n,q)=1$ (see e.g. \cite[Proposition 10.1.1]{CoxLittleSchenck}). We will assume $n>q>0$, since otherwise $r=0$, and Proposition \ref{lem:sup:pretiltingIU} holds trivially.
Let
\[
\frac{n}{q}=[a_1,\dots,a_r]:=a_1-\dfrac{1}{a_2-\dfrac{1}{\cdots-\dfrac{1}{a_r}}}
\]
be the Hirzebruch-Jung continued fraction expansion of $n/q$ where $a_i\ge 2$. 
The $a_j$, $j\geq 1$, may be  computed by an Euclidean style algorithm starting from $i_0=n$, $i_1=q$, 
and inductively determining $a_j,i_{j+1}$ via division with remainder:
\begin{equation}\label{i}
i_{j-1}=a_{j}i_{j}-i_{j+1}, \quad 0<i_{j+1}<i_{j}, 
\end{equation}
until we obtain $i_{r}=1$. Then we set  $a_r=i_{r-1}$ and $i_{r+1}=0$. 

 The vectors $u_0,u_1,\dots,u_{r+1}$ 
 are related by the equations
\begin{equation}\label{012}
a_j u_j=u_{j-1}+u_{j+1},
\end{equation}
$1\leq j\leq r$, (see \cite[Theorem 10.2.8(b), Theorem 10.2.5]{CoxLittleSchenck}).

Below we extend $[a_1,\ldots,a_r]$  to a doubly infinite sequence $(a_i)_{i\in\ZZ}$ by putting $a_i=0$ for $i\not\in \{1,\ldots,r\}$.
We let $q_{s,t}$ be as in \S\ref{sec:recursion}.
\begin{remark}
\label{rem:recursion}
By \eqref{i}, \eqref{012} $x_l=i_{l}$, $x_l=u_{l}$ are solutions to \eqref{eq:recursion}. We silently extend them
to doubly infinite solutions  to avoid having to keep track of cumbersome restrictions on indices in some formulas.
\end{remark}

\begin{remark}
The $q_{s,t}$, for suitable $s,t$, appear frequently in combinatorics of Hirze\-bruch-Jung continued fractions. 
In \cite[Proposition 10.2.2]{CoxLittleSchenck}, $q_{0,t+1}$, $q_{1,t+1}$  are denoted by $P_t$, $Q_t$, respectively, for $0\leq t\leq r$. 
For example, they provide expressions for the truncated Hirzebruch-Jung continued fractions: $[b_{s},\dots,b_t]=q_{s-1,s+t}/q_{s,s+t}$ for $1\leq s\leq t\leq r$. The formula is proved in loc. cit.  for right truncations, i.e. for $s=1$, and it is not difficult to see that it holds also for general truncations.
\end{remark}

\subsubsection{Application}
\begin{lemma}
\label{lem:ordering}
Assume $0\leq s\leq s'\leq t'\leq t\leq r+1$ and $(s',t')\neq (s,t)$. Then 
\begin{equation}
\label{eq:order1}
0\le q_{s',t'}<q_{s,t}.
\end{equation}
\end{lemma}
\begin{proof} 
  This follows easily by induction on $s,t$  using the fact that $a_i\ge 2$ for $i=1,\ldots,r$.
\end{proof}
\begin{lemma}
We  have 
\begin{equation}
\label{0jm}
u_{j}=\frac{q_{j,l}}{q_{j-1,l}}\,u_{j-1}+\frac{1}{q_{j-1,l}}
\,u_l
\end{equation}
for $1\leq j\leq l\le r+1$.
\end{lemma}
\begin{proof}
By Remark \ref{rem:recursion} and Lemma \ref{lem:specialized}, $C_l:=q_{j-1,l}u_j-q_{j,l}u_{j-1}$ does not depend on $j$. Specializing to $j=l$ we see that $C_l=u_{l}$.
\end{proof}
Substituting $u_{j}=v_j-v_{-1}$ in \eqref{0jm} 
we obtain 
\begin{equation}
\label{eq:inductiondatum10}
v_j=\frac{q_{j-1,l}-q_{j,l}-1}{q_{j-1,l}}
\,v_{-1}+\frac{q_{j,l}}{q_{j-1,l}}\,
v_{j-1}+\frac{1}{q_{j-1,l}}
\,v_{l},
\end{equation}
which gives a concrete expression for the coefficients in  \eqref{eq:inductiondatum0} (and also for the induction datum in Proposition \ref{prop:main} below). 
Note that we have
\begin{equation}
\label{eq:i}
i_{t}=q_{t,r+1}
\end{equation}
as both sides satisfy \eqref{eq:recursion} by Remark \ref{rem:recursion} and are equal for $t=r,r+1$. 
Substituting this in \eqref{eq:inductiondatum}
we obtain
\begin{equation}
\label{eq:inductiondatum}
v_j=\frac{i_{j-1}-i_j-1}{i_{j-1}}\,v_{-1}+\frac{i_j}{i_{j-1}}\,v_{j-1}+\frac{1}{i_{j-1}}\,v_{r+1},
\end{equation}
and hence by \eqref{eq:defbj0}
\begin{equation}\label{eq:defbj}
b_{j}=\bigg\lf\frac{i_{j-1}-i_j-1}{i_{j-1}}\,b_{-1}+\frac{i_j}{i_{j-1}}\,b_{j-1}+\frac{1}{i_{j-1}}\,b_{r+1}\bigg\rf.
\end{equation}

\subsection{An expression for $\boldsymbol{b}$}
We give an explicit expression for the solutions to \eqref{eq:defbj} in case $(b_{-1},b_0,b_{r+1})=(0,0,d)$, $0\le d<n$
(see Lemma \ref{lem:pom2} below with $l=r+1$).  Since $i_1>\cdots>i_r=1$ 
successive division with remainder yields a unique representation
\begin{equation*}
\label{eq:ddef}
d=\sum_{t=1}^r i_td_t
\end{equation*}
 with $d_t\in \NN$ and
\begin{equation}
\label{eq:dprop}
 \sum^r_{t=k+1}i_td_t<i_{k}\text{ for any }0\le k\leq r 
\end{equation}
(for $k=0$ this formula becomes $d<n$ which is true by hypothesis).

For $0\leq j\leq r+1$  put
\begin{equation}\label{eq:bp} 
c_{j}=\sum_{t=1}^{j-1} q_{t,j}d_t.
\end{equation}
Using \eqref{eq:i} and interpreting an empty sum as zero we have in particular
\begin{equation}
\label{eq:boundary}
c_0=0, \quad c_{r+1}=d.
\end{equation}
There is bound for the tail
sums in \eqref{eq:bp} similar to \eqref{eq:dprop}.
\begin{lemma} We have for $0\le k\le j-1\le r$
\begin{equation}\label{des}
\sum_{t=k+1}^{j-1}q_{t,j}d_t
<q_{k,j}.
\end{equation}
\end{lemma}
\begin{proof}
For $k=j-1$ there is nothing to prove so we assume $k<j-1$. According to  \cite[Lemma 3.5]{IshiiUeda} $d_t$ has the following properties
\begin{enumerate}
\item[(i)] $0\le d_t\le a_t-1$.
\item[(ii)] If $d_s=a_s-1$, $d_t=a_t-1$ for $s<t$ then there is $s<l<t$ such that $d_l=a_l-3$.
\end{enumerate}

By \eqref{eq:order1} the sequence $(q_{t,j})_t$ is descending. Using this it is easy to see that  among the sequences $(d_t)_t$ which satisfy (i), (ii) the one that gives the
maximum value for $\sum_{t=k+1}^{j-1}q_{t,j}d_t$ is given by $d_{k+1}=a_{k+1}-1$, $d_{k+2}=a_{k+2}-2$,\dots,$d_{j-1}=a_{j-1}-2$.
Hence the left-hand side of \eqref{des} is bounded by
\[
q_{k+1,j}+\sum_{t=k+1}^{j-1}q_{t,j}(a_t-2).
\]
We compute
\begin{align*}
q_{k+1,j}+\sum_{t=k+1}^{j-1}
q_{t,j}(a_t-2)
&=q_{k+1,j}+\sum_{t=k+1}^{j-1}\left(q_{t-1,j}-q_{t,j}+q_{t+1,j}-q_{t,j}\right)&&\eqref{p012}\\\nonumber
                          &=q_{k+1,j}+q_{k,j}-q_{j-1,j}-q_{k+1,j}+q_{j,j}\\\nonumber
&=q_{k,j}-1
\end{align*}
which implies \eqref{des}.\end{proof}
\begin{lemma} \label{lem:pom2} We have for $1\leq j\le l\leq r+1$
\begin{align}\label{eq:pom2}
c_{j}
&=\left\lfloor 
\frac{q_{j,l}}{q_{j-1,l}}c_{j-1}+\frac{1}{q_{j-1,l}}c_l\right \rfloor.
\end{align}
\end{lemma}
\begin{proof}
We compute the right-hand side of \eqref{eq:pom2}. 
We have
\begin{align*}
\frac{q_{j,l}}{q_{j-1,l}}c_{j-1}+\frac{1}{q_{j-1,l}}c_l
&=\frac{q_{j,l}}{q_{j-1,l}}\sum_{t=1}^{j-2}q_{t,j-1}d_t+
  \frac{1}{q_{j-1,l}}\sum_{t=1}^{l-1}q_{t,l}d_t&\eqref{eq:bp}\\
&=\sum_{t=1}^{j-2}\frac{q_{j,l}q_{t,j-1}+q_{t,l}}{q_{j-1,l}}d_t+d_{j-1}+\frac{1}{q_{j-1,l}}\sum_{t=j}^{l-1} q_{t,l}d_t\\
&=\sum_{t=1}^{j-1}q_{t,j}d_t+\frac{1}{q_{j-1,l}}\sum_{t=j}^{l-1} q_{t,l}d_t&\eqref{eq:recursion3}.
\end{align*}
Using \eqref{eq:bp}, \eqref{des} we see that the right-hand side of \eqref{eq:pom2}  equals $c_{j}$. 
\end{proof}

\subsection{Proof of Proposition \ref{lem:sup:pretiltingIU}}
\label{sec:combred}
Assume $b=(b_{-1},b_0,b_{r+1},b_1,\dots,b_r)$ is obtained by triangle convex induction from $(b_{-1},b_0,b_{r+1})$. Then it is trivially true
that for $0\le l\le r$, $(b_{-1},b_{l},b_{r+1},b_{l+1},\allowbreak \ldots,b_{r})$ is obtained by triangle convex
induction from $(b_{-1},b_{l},b_{r+1})$.

The following result shows that this is also true when we ``truncate'' $b$  on the right. 
This is a combinatorial fact for which we have no conceptual explanation.

\begin{proposition} \label{prop:main}
For $1\le l\le r+1$,
$(b_{-1},b_0,b_l,b_1,\ldots,b_{l-1})$ is obtained by triangle convex induction from
$(b_{-1},b_0,b_{l})$. 
\end{proposition}

\begin{proof}
In view of \eqref{eq:inductiondatum10} we have to prove for $1\leq j<l\leq r+1$
\begin{align}\label{eq:pom1}
b_{j}
&=\left\lfloor \frac{q_{j-1,l}-q_{j,l}-1}{q_{j-1,l}}\,b_{-1}+
\frac{q_{j,l}}{q_{j-1,l}}b_{j-1}+\frac{1}{q_{j-1,l}}b_l\right \rfloor.
\end{align}
To prove \eqref{eq:pom1} we may add to $b$ an affine function of the form $\langle m,-\rangle+k$ for $m\in M$ and $k\in\ZZ$
as such an altered $b$ still satisfies the recursion  \eqref{eq:defbj}. Using \eqref{eq:us} it is therefore easy to see that we may assume $(b_{-1},b_0,b_{r+1})=(0,0,d)$ for $0\le d<n$.

Comparing \eqref{eq:defbj} (and \eqref{eq:i}) with \eqref{eq:pom2} for $l=r+1$ together with \eqref{eq:boundary} we find that the normalized $b$ satisfy $b_j=c_j$. Then \eqref{eq:pom1} follows by invoking
 \eqref{eq:pom2} for arbitrary~$l$.
\end{proof}
\begin{remark} 
\label{rem:sign} It is clear that in the proof of Proposition \ref{prop:main}  we made essential use of the fact that in \eqref{eq:defbj} we used
$\lf-\rf$, rather than $\lc-\rc$. 
\end{remark}

\begin{proof}[Proof of Proposition \ref{lem:sup:pretiltingIU}]
It is enough to prove that $s_{v_{-1}}^c(m)=s_{v_{j-1}}^c(m)=s_{v_l}^c(m)\allowbreak \neq s_{v_{j}}^c(m)$, $1\leq j<l$, cannot occur. 
This follows from 
 Lemmas \ref{lem:sign}, \ref{diff} (using Proposition \ref{prop:main}).
\end{proof}

\section{A tilting bundle on a stacky resolution using Gulotta's approach}\label{Gulotta}
From now on we assume that $n=3$.  We fix a triangulation of $P$ without extra vertices and we let $\Sigma$ be the corresponding fan. In this section
we will construct a tilting bundle on $X_{\bold{\Sigma}}$. This yields an NCCR for $R_P$ by Proposition \ref{sec:standard}.

To start we embed $P$  in a (minimal) rectangle 
$P_0=[0,\bc]\times [0,\bd]\times\{1\}$.  
Then $P_0-\text{int}P$ is the union of at most four (non-convex) polygons $\{Q_i\}$, one polygon for every corner vertex in $P_0$ not occurring in $P$. Each of these polygons can be divided  into 
triangles by successively cutting off ``corner triangles'' of maximal size whose internal edge has slope determined by the Farey tree, traversed row by row (see \cite[Figure 15]{Gulotta}). E.g.\ if $Q\in \{Q_i\}$ is the polygon in the 
upper left or lower right corner then we use the slope sequence\footnote{The slopes $0/1$ and $1/0$ which are sometimes considered part of  the Farey tree are missing as they correspond to the vertical and horizontal boundary of $Q$.}
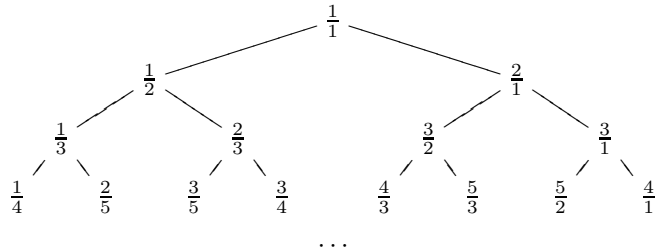
\[
\xymatrix@=1ex{
&&&&&&&&\frac{1}{1}\ar@{-}[dllll]\ar@{-}[drrrr]&&&&&&&&\\
&&&&\frac{1}{2}\ar@{-}[dll]\ar@{-}[drr]&&&&&&&&\frac{2}{1}\ar@{-}[dll]\ar@{-}[drr]&&&&\\
&&\frac{1}{3}\ar@{-}[dl]\ar@{-}[dr]&& &&\frac{2}{3}\ar@{-}[dl]\ar@{-}[dr]&& &&\frac{3}{2}\ar@{-}[dl]\ar@{-}[dr]&& &&\frac{3}{1}\ar@{-}[dl]\ar@{-}[dr]&&\\
&\frac{1}{4}& &\frac{2}{5}& &\frac{3}{5}& &\frac{3}{4}& &\frac{4}{3}& &\frac{5}{3}& &\frac{5}{2}& &\frac{4}{1}&\\
&&&&&&&&\cdots&&&&&&&&\\
}
\]
If $Q\in \{Q_i\}$ is the polygon in the upper right or lower left corner then we use the corresponding negative slopes. See Figure \ref{fig:gulotta} for an example.

\begin{figure}[H]
\begin{tikzpicture}[scale=1.00]
\path[draw,color=lightgray] (0.00,0.00) grid (4.00,3.00);
\path[draw,line width=0.3mm] (0.00,0.00)-- (4.00,0.00)-- (4.00,1.00)-- (3.00,2.00)-- (1.00,3.00)-- (0.00,0.00);
\path[draw,color=ruta1,] (2.00,3.00) -- (3.00,2.00);
\path[draw,color=ruta1,] (1.00,3.00) -- (0.00,2.00);
\path[draw,color=ruta1,] (1.00,3.00) -- (0.00,1.00);
\path[draw,color=siva,] (4.00,1.00) -- (0.00,0.00);
\path[draw,color=siva,] (3.00,2.00) -- (0.00,0.00);
\path[draw,color=red,] (4.00,3.00) -- (3.00,2.00);
\node[] at (3.50,3.30) {$P_0$};
\node[] at (3.70,1.70) {$P$};
\tkzDefPoint(0.0,0.0){a}
\tkzDefPoint(0.0,1.0){b}
\tkzDefPoint(1.0,3.0){c}
\tkzDefPoint(0.0,2.0){d}
\tkzDefPoint(0.0,3.0){e}
\tkzDefPoint(2.0,3.0){f}
\tkzDefPoint(4.0,3.0){g}
\tkzDefPoint(3.0,2.0){h}
\tkzDefPoint(4.0,1.0){i}
\tkzDefPoint(4.0,0.0){j}
\tkzFillPolygon[fill=ruta,opacity=0.5](a,b,c)
\tkzFillPolygon[fill=ruta,opacity=0.5](b,c,d)
\tkzFillPolygon[fill=ruta,opacity=0.5](d,c,e)
\tkzFillPolygon[fill=ruta,opacity=0.5](c,f,h)
\tkzFillPolygon[fill=ruta,opacity=0.6](f,g,h)
\tkzFillPolygon[fill=ruta,opacity=0.6](h,i,g)
\tkzFillPolygon[fill=siva,opacity=0.5](a,i,j)
\tkzFillPolygon[fill=siva,opacity=0.5](a,i,h)
\tkzFillPolygon[fill=siva,opacity=0.5](a,h,c)
\end{tikzpicture}

\caption{\textcolor{ruta}{$\blacktriangle$} corner triangles, \textcolor{red}{\line(1,1){10}}extra diagonal} 
\label{fig:gulotta}
\end{figure}

Let $P_i$ be the polygon that we obtain after removing $i$ triangles from $P_0$
 and let $P_l=P$. As explained above $P_i$ is obtained by ``sawing off'' a corner from $P_{i-1}$. The thus created edge in $P_i$ will be called the \emph{cutting edge}.  Remarkably, because
of the clever slope choices during the cutting, all $P_i$ are lattice polygons (cf. proof of
\cite[Theorem 6.1]{Gulotta}).

 We add extra diagonals to the triangle $P_{i-1}-P_i$ connecting the unique vertex of $P_{i-1}$ not in $P_i$  to the points on the cutting edge in $P_i$ which are vertices of $P_j$ for $j\ge i$.%
\footnote{The extra diagonals are necessary to obtain a triangulation of $P_0$ in the sense of algebraic topology, i.e. a homeomorphism $|K|\cong P_0$ where $K$ is a combinatorial simplicial
complex.}
Combining the resulting triangulation of $P_i-P$ with the one of $P$ we obtain compatible triangulations of $P_0\supset P_1\supset\cdots\supset P_l=P$
where $P_{i-1}-P_{i}$ is a subdivided triangle. Let $V_0\supset V_1\supset\cdots\supset V_l$ be the corresponding sets of vertices. Note that $|V_{i-1}-V_i|=1$ and moreover $V_i$
is the union of the vertices of $P_j$ for $j\ge i$.

Let $n_1=(0,0,1)$, $n_2=(\bc,0,1)$, $n_3=(\bc,\bd,1)$, $n_4=(0,\bd,1)$ be the vertices of~$P_0$
and
let $V'_i$ be the union of the vertices of $P_j$ for $0\le j\le i$.
For each $i$,  $V'_i-V'_{i-1}$ consists of endpoints of the cutting edge
in $P_i$. More precisely $V'_i-V'_{i-1}$ consists of $0$, $1$ or $2$ elements of $\partial P_{i-1}$ (depending on how many of the endpoints of the cutting edge are already vertices of $P_{i-1}$). If $v$ is one of those elements then 
\begin{equation}
\label{eq:conv2}
v=(1-t)v'+tv''
\end{equation}
for $t\in ]0,1[$ and $v',v''$ vertices of $P_{i-1}$. We
use this to order $V_0=V'_l=[n_1{},\ldots,n{}_{4+l'}]$ in
such a way that it is convexly induced from $[n_1,n_2,n_3,n_4]$ in the
sense of Definition \ref{def:convex1}, with the induction datum $\mathfrak{t}=(t^{(i)}_j)_{i,j}$
being obtained from the expressions \eqref{eq:conv2}.

We now fix $b_0\in \ZZ^4$. 
We also fix a sign sequence and we let $(b_v)_{v\in V_0}:=\tilde{b}{}$ be \convind from $b_0$ with respect to this sign sequence (and the induction datum $\mathfrak{t}$). 
In down to earth terms this means that whenever $v,v',v''$ are as in \eqref{eq:conv2}
then $b_{v}=\lfloor (1-t)b_{v'}+tb_{v''}\rfloor$ or $b_{v}=\lceil (1-t)b_{v'}+tb_{v''}\rceil$
depending on the chosen sign sequence.

Let $\Sigma_i$ be the fan on the cone over $P_i$ corresponding to its triangulation. The elements of $V_i$ are the minimal generators of the one-dimensional cones in $\Sigma_i$ and we 
 let ${\bf \Sigma_i}=(\Sigma_i,V_i)$ be the corresponding stacky fan as in \S\ref{sec:prel}.
Let $\tilde{b}_i=(b_v)_{v\in V_i}$.
\begin{lemma}\label{lem:pretilting}
We have
\begin{equation}
\label{eq:gulotta}
\Gamma(X_{\bold{\Sigma}_0},\Mscr_{\bold{\Sigma}_0,\tilde{b}})=\M_{b_0}.
\end{equation}
If $b_0\in \ZZ^4$ moreover satisfies
\begin{equation}{\cm}\label{\cm}
\forall m\in M:[s_1^{b_0}(m),s_2^{b_0}(m),s_3^{b_0}(m),s_4^{b_0}(m)]\not\in \{[+-+-],[-+-+]\}
\end{equation}
(cfr. \eqref{eq:signdef} for notation) then $H^j(X_{\bold{\Sigma}_i},\Mscr_{\bold{\Sigma}_i,\tilde{b}_i})=0$ for $j>0$
and $i=0,\ldots,l$.
\end{lemma}
\begin{proof}
We  start with the first claim. By Lemma \ref{lem:relations} we have
\[
\Gamma(X_{\bold{\Sigma}_0},\Mscr_{\bold{\Sigma}_0,\tilde{b}})=\M_{\tilde{b}}.
\]
Since $\tilde{b}$ is convexly induced from $b_0$, the conclusion follows from Corollary \ref{convind}.

Now we discuss the last claim. As $\pi_s:X_{{\bf \Sigma}_i}\to X_{\Sigma_i}$  
 is finite we have by Lemma \ref{lem:relations}
\[
H^j(X_{\bold{\Sigma}_i},\Mscr_{\bold{\Sigma}_i,\tilde{b}_i})=H^j(X_{\Sigma_i},\Mscr_{\Sigma_i,\tilde{b}_i})=H^j(X_{\Sigma_i},\Oscr(D_{\tilde{b}_i})).{}
\]

 We will prove that $V_{D_{\tilde{b}_i},m}$ (see \eqref{eq:cohomMb}, \eqref{vdbm}) is either empty or contractible.\footnote{Note that the distinction is relevant as contractible spaces have no reduced cohomology but $\tilde{H}^{-1}(\emptyset,k)=k$.} This implies what we want. 
\begin{step} \label{step1}
Let $u_0,u_1,\ldots, u_{r+1}\in V_i$ be the vertices on an edge in $P_i$
(ordered in an arbitrary direction). We claim that the signs $s_{u_i}^b(m)$ follow the 
pattern $+\cdots+-\cdots-$ (possibly reflected and perhaps with no $+$ or $-$ present). 
This follows from the easy observation that $(\tilde{b}_{u_0},\tilde{b}_{u_{r+1}},\tilde{b}_{u_1},\ldots,\tilde{b}_{u_r})$ is obtained
by interval convex induction from $(\tilde{b}_{u_0},\tilde{b}_{u_{r+1}})$ together with Lemma \ref{lem:intervalconvexinduction}.
\end{step}
 \begin{step} \label{step2} Let $w_1,w_2,\ldots$ be the elements of $\partial P_i\cap V_i$ (ordered in 
   an arbitrary direction). We claim that the signs $s_{w_i}^b(m)$ follow the 
pattern $+\cdots+-\cdots-$ (up to cyclic permutation and possibly with no $+$ or $-$ present).

By Step \ref{step1} it is sufficient to prove this for $w_1,w_2,\ldots$ being the vertices of $P_i$.
We note that it is true for $P_0$ by \eqref{\cm}.
An easy case by case analysis using  Lemma \ref{lem:sign} 
for removing a single triangle 
 then proves the claim by induction.
\end{step}
\begin{step} $V_{D_{\tilde{b}_l},m}$ is either empty or contractible. This follows from Step \ref{step2} applied with $i=l$. We leave the verification as an exercise to the reader. 
\end{step}
\begin{step} \label{step4} If the sign pattern in Step \ref{step2} is $+\cdots+$ then $V_{D_{\tilde{b}_i},m}=\emptyset$. Indeed, it follows
from the inductive construction and repeated application of Lemma \ref{lem:sign} that $s_{v}^b(m)=+$ for all $v\in V_i$ in this case.
\end{step}
\begin{step} \label{step5} Again by a case by case analysis of the same nature as in Step 
\ref{step2} one shows that the following possibilities can occur for the relation between $V_{D_{\tilde{b}_i},m}$ and $V_{D_{\tilde{b}_{i-1},m}}$.
\begin{enumerate}
\item They are the same.
\item $V_{D_{\tilde{b}_i},m}=\emptyset$ and $V_{D_{\tilde{b}_{i-1},m}}$ is a single point (this needs in particular Step \ref{step4}). 
\item $V_{D_{\tilde{b}_{i-1},m}}$ is obtained from $V_{D_{\tilde{b}_i},m}$ by gluing an interval at an end point.
\item $V_{D_{\tilde{b}_{i-1},m}}$ is obtained from $V_{D_{\tilde{b}_i},m}$ by gluing a 
 triangle at one of its edges. 
\end{enumerate}
In each of the cases the property of being empty or contractible is preserved, finishing the proof.\qedhere
\end{step}
\end{proof}

We set 
\[
S_0=\{(0,i,i+j,j)\mid 0\leq i<\bc,0\leq j<\bd\}\cup \{(0,i,i+j+1,j)\mid 0\leq i<\bc,0\leq j<\bd\}.
\] 

\begin{remark}\label{rmk:nccr0} 
Note that $X_{P_0}=Y\quot G$ where $Y=(k^{\oplus 4})^\vee$ and $G=k^*\times \ZZ_\bc\times \ZZ_\bd$ acts on $k^{\oplus 4}$ with weights $\chi_1\xi_\bc^{-1}\xi_\bd^{-1},\chi_{-1}\xi_\bc,\chi_1,\chi_{-1}\xi_\bd$ (see \S\ref{sec:prel}). Furthermore this action is generic in the sense of  \cite[Def.\S 1.3.4]{SVdB}.
 Since $S_0$ can be naturally identified with $\{0,1\}\times\ZZ_\bc\times \ZZ_\bd=((1/2)[-2,2]_\epsilon\cap X(k^*))\times X(\ZZ_\bc\times\ZZ_\bd)$ for $\epsilon>0$, it follows by \cite[Theorem 1.6.2]{SVdB} combined with \cite[Lemma 4.5.1]{SVdB} (see also \cite[\S 7]{VdB32}) that $\Lambda_0:=\End_{R_{P_0}}(\bigoplus_{b\in S_0} \M_b)$ is a (well-known) NCCR of $R_{P_0}$. 
\end{remark}
For each vertex in $V_0$ we choose a sign  and  
we compatibly \convinde the elements $S_0=\{(b_v)_{v\in V'_0}\}$ using those signs 
to obtain a set $S=S_l=\{(b_v)_{v\in V_0}\}$ ($S_l$ and $S_0$ are in bijection). 

Let $\ol{S}_l^i\subset \prod_{v\in V_i}\ZZ$ be the image of  $S_l$  under the projection $\prod_{v\in V_0}\ZZ\r \prod_{v\in V_i}\ZZ$. 
Put $\Tscr_i=\bigoplus_{b\in \ol{S}^i_l} \Mscr_{\bold{\Sigma}_i,b}$ (see \eqref{eq:can} for notation).
This is a vector bundle on the toric DM-stack $X_{\bf \Sigma_i}$.
\begin{corollary}\label{cor:pretilting}
$\Tscr_i$ has no higher self-extensions. 
\end{corollary}

\begin{proof} 
It is easy to verify that $b-b'$ for $b,b'\in S_0$ satisfy \eqref{\cm}.  As $S_l$ is compatibly \convind from $S_0$, Lemma \ref{diff} enables us to apply Lemma \ref{lem:pretilting}. 
\end{proof}

\begin{lemma}\label{tilting}
$\Tscr_0$ is a classical generator (see \cite{BondalVdB}) of $\Dscr(X_{\bf \Sigma_0})$. 
\end{lemma}
\begin{proof}
By 
\eqref{eq:gulotta}, Lemma \ref{diff} and Remark \ref{rmk:nccr0} (combined with \eqref{eq:refhom}) 
we obtain $\tilde{\Lambda}_0:=\End_{X_{\bf \Sigma_0}}(\Tscr_0)\cong \Lambda_0$.

Since $\Lambda_0$ is an NCCR for $R_{P_0}$ by Remark \ref{rmk:nccr0}, 
 we have in particular $\gldim \Lambda_0<\infty$, and hence the same is true for $\tilde{\Lambda}_0$.
By Corollary~\ref{cor:pretilting} we have a full and faithful embedding of $\Perf(\tilde{\Lambda}_0)$ in $\Dscr(X_{\bf\Sigma_0})$. 
As $\gldim \tilde{\Lambda}_0<\infty$, we have $\Perf(\tilde{\Lambda}_0)=\Dscr(\tilde{\Lambda}_0)$ and moreover $\Dscr(\tilde{\Lambda}_0)$ is admissible in $\Dscr(X_{\bf\Sigma_0})$. 
Since $\Dscr(X_{\bf\Sigma_0})$ does not have a nontrivial semi-orthogonal decomposition by Corollary \ref{cor:sod}, we have $\Dscr(\tilde{\Lambda}_0)\cong \Dscr(X_{\bf\Sigma_0})$. 
\end{proof}

\begin{theorem}\label{itilting}
$\Tscr_i$ is a tilting bundle on $X_{{\bf \Sigma}_i}$.
\end{theorem}

\begin{proof}
By Corollary \ref{cor:pretilting} it suffices to show that $\Tscr_i$ is a classical generator of $\Dscr(X_{{\bf \Sigma}_i})$. Since $X_{{\bf \Sigma}_i}$ is isomorphic to an open substack of 
$X_{{\bf \Sigma}_0}$ and $\Tscr_i$ is a restriction of $\Tscr_0$ to $X_{\bf\Sigma_i}$ this follows by Lemma \ref{tilting}.
\end{proof}

\section{Tilting bundle on a stacky resolution using the Ishii-Ueda approach}\label{IshiiUeda}
We start in exactly the same way as in \S\ref{Gulotta}. In this section we will construct a tilting bundle on $X_{\bold{\Sigma}}$
using a procedure inspired by \cite{IshiiUeda}. This will yield again an NCCR of $R_P$ using Proposition \ref{sec:standard}.

Now we embed $P$ in a triangle $P_0$ with vertices $(0,0,1),(0,\bc,1),(\bd,0,1)$. 
To obtain $P_i$, $i>0$, from $P_{i-1}$ one removes a (corner) vertex of $P_{i-1}$ and takes the convex hull of the remaining lattice points in $P_{i-1}$. Let $l$ be such that $P_l=P$. 
 We triangulate $P_0$ as in \S\ref{Gulotta} by first triangulating $P$ (without extra vertices) and then triangulating 
$P_{i-1}-P_{i}$ by adding edges from the unique vertex in $P_{i-1}$ which is not on $P_i$ to the vertices of $P_{j}$, $j\ge i$ which are not on $\partial P_{i-1}$. See Figure \ref{fig:iu}.

\begin{figure}[H]
\begin{tikzpicture}[scale=1.00]
\path[draw,color=lightgray] (0.00,0.00) grid (4.00,1.00);
\path[draw,color=lightgray] (0.00,1.00) grid (3.00,2.00);
\path[draw,color=lightgray] (0.00,2.00) grid (2.00,3.00);
\path[draw,color=lightgray] (0.00,3.00) grid (1.00,4.00);

\path[draw,line width=0.3mm] (0.00,0.00)-- (4.00,0.00)-- (4.00,1.00)-- (3.00,2.00)-- (1.00,3.00)-- (0.00,0.00);
\node[] at (0.70,4.70) {$P_0$};
\node[] at (3.70,1.70) {$P$};
\tkzDefPoint(0,0){a};
\tkzDefPoint(4,0){b};
\tkzDefPoint(5,0){c};
\tkzDefPoint(4,1){d};
\tkzDefPoint(3,2){e};
\tkzDefPoint(2,3){f};
\tkzDefPoint(1,3){g};
\tkzDefPoint(0,3){h};
\tkzDefPoint(1,4){i};
\tkzDefPoint(0,5){j};
\tkzDefPoint(0,4){k};
\tkzDefPoint(0,2){l};
\tkzDefPoint(0,1){m};
\path[draw,color=ruta1,](0,3)--(1,4);
\path[draw,color=ruta1,](0,0)--(1,4);
\path[draw,color=ruta1,](0,5)--(0,3);
\path[draw,color=ruta1,](0,5)--(1,4);
\path[draw,color=ruta1,](1,3)--(1,4);
\path[draw,color=ruta1,](2,3)--(1,4);
\path[draw,color=ruta1,](1,3)--(2,3);
\path[draw,color=ruta1,](2,3)--(3,2);
\path[draw,color=ruta1,](4,0)--(5,0);
\path[draw,color=siva,](0,0)--(3,2);
\path[draw,color=siva,](0,0)--(4,1);
\path[draw,color=ruta1,](5,0)--(4,1);
\path[draw,color=ruta1,](0,0)--(0,3);
\path[draw,color=ruta1,](0,4)--(1,4);
\path[draw,color=ruta1,](0,2)--(1,4);
\path[draw,color=ruta1,](0,1)--(1,4);
\tkzDrawPoint[color=ruta1,fill=ruta,size=10](j)
\tkzDrawPoint[color=ruta1,fill=ruta,size=10](h)
\tkzDrawPoint[color=ruta1,fill=ruta,size=10](i)
\tkzDrawPoint[color=ruta1,fill=ruta,size=10](f)
\tkzDrawPoint[color=ruta1,fill=ruta,size=10](c)
\tkzDrawPoint[color=ruta1,fill=ruta,size=10](k)
\tkzDrawPoint[color=ruta1,fill=ruta,size=10](l)
\tkzDrawPoint[color=ruta1,fill=ruta,size=10](m)
\tkzFillPolygon[fill=siva,opacity=0.5](a,g,e)
\tkzFillPolygon[fill=siva,opacity=0.5](a,e,d)
\tkzFillPolygon[fill=siva,opacity=0.5](a,d,b)
\tkzFillPolygon[fill=ruta,opacity=0.35](j,k,i)
\tkzFillPolygon[fill=ruta,opacity=0.43](k,h,i)
\tkzFillPolygon[fill=ruta,opacity=0.5](l,h,i)
\tkzFillPolygon[fill=ruta,opacity=0.57](l,m,i)
\tkzFillPolygon[fill=ruta,opacity=0.65](a,m,i)
\tkzFillPolygon[fill=ruta,opacity=0.75](a,g,i)
\tkzFillPolygon[fill=ruta,opacity=0.75](i,g,f)
\tkzFillPolygon[fill=ruta,opacity=0.87](g,f,e)
\tkzFillPolygon[fill=ruta,opacity=0.3](b,c,d)
\end{tikzpicture}

\caption{\textcolor{ruta}{\Large$\bullet$} corner vertices}
\label{fig:iu}
\end{figure}
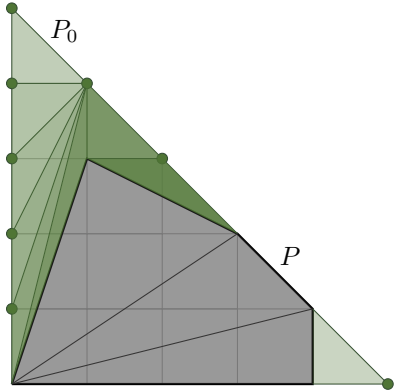

The objects $\Sigma_i$, ${\bf \Sigma_i}$, $V_i$, $V_i'$ are defined in the same way as in \S\ref{Gulotta}. 
We set 
\begin{align*}
S_0&=\{(i,j,1)\mid 0\leq i<\bc,0\leq j<\bd\}.
\end{align*} 
Note that $X_{P_0}\cong Y\quot G$, where $Y=k^3$, $G=\ZZ_\bc\times \ZZ_\bd$ acts on $k^3$ with weights $\xi_\bc^{-1}\xi_\bd^{-1},\xi_\bc,\xi_\bd$ (see \S\ref{sec:prel}). This action is generic. 
Since $G$ is a finite abelian group and $S_0\cong X(G)$, it is classical (see e.g.\  \cite{Auslander}, \cite[\S J]{Leuschke}, \cite[\S1.2]{SVdB}, \cite[\S1.4]{Wemyss}) that 
 $\Lambda_0=\End_{R_{P_0}}(\bigoplus_{b\in S_0} \M_b)$ is an NCCR of $R_{P_0}$. 

Below we will inductively extend $S_{i-1}=\{(b_i)_{i\in V'_{i-1}}\}$ to $S_i=\{(b_i)_{i\in V'_i}\}$ using compatible
convex induction so that the required tilting bundle will be obtained from $S_l=\{(b_i)_{i\in V_0}\}$ (note $V'_l=V_0$).
However here the induction process appears to be more delicate and in particular our proof depends on the choice
of a specific sign sequence (see \S\ref{subsec:convind}) in order to obtain an analogue for Lemma \ref{lem:pretilting}. See Remark \ref{rem:sign}.
\subsection{Definition of\mathversion{bold} $S_i$}
\label{sec:Si}
Let $v_{-1}\in V'_{i-1}$ be the vertex of $P_{i-1}$ that we remove to construct $P_i$, and let $v_0,v_{r+1}\in \partial P_{i-1}\cap\partial P_i\cap \partial(P_{i-1}\setminus P_{i})$.  
Let $v_1,\dots,v_r$ be the other lattice points 
on $\partial(P_{i-1}\setminus P_{i})$  such that $v_0,v_1,\dots,v_{r},v_{r+1}$ are consecutive. 
\begin{remark}
Note that $V'_{i-1}\cup\{v_0,v_1,\dots,v_{r+1}\}$ may be strictly larger than $V'_i$ since some of the $v_j$ may not be vertices of  $P_i$ (they will of course be elements of
$\partial P_i$). For a simple example how this can happen
look at \cite[Figure 10.3]{CoxLittleSchenck}.
\end{remark}

Let $b\in S_{i-1}$. We describe how to induce $b$ to $V'_i$. This will be done for all $b\in S_{i-1}$ in a compatible way.
\begin{enumerate}[(i)]
\item \label{ind1}
We first \convinde $b$ to $V'_{i-1}\cup \{v_0,v_{r+1}\}$ for a fixed sign sequence. We keep denoting the result by $b$. 
\item We induce up $b$ to $V'_{i-1}\cup \{v_0,\ldots,v_{r+1}\}$ by triangle convex induction (see \S\ref{sec:triangle}),  and set    
$b^{(i)}=(b_v)_{v\in V'_{i}}$. We let $S_i$ be the collection of $b^{(i)}$ so obtained.
\end{enumerate}
We also adopt the notation  $\ol S^i_l$ from \S\ref{Gulotta}. We can now state the following analogue of Theorem \ref{itilting}.
\begin{theorem}\label{iuitilting}
Put $\Tscr_i^{\triangle}:=\bigoplus_{b\in \ol{S}^i_l} \Mscr_{\bold{\Sigma}_i,b}$. 
Then $\Tscr_i^{\triangle}$ is a tilting bundle on $X_{{\bf \Sigma}_i}$.
\end{theorem}
This is proved using Proposition \ref{lem:pretiltingIU} below, which allows us to suitably adapt the proof of Corollary \ref{cor:pretilting}, in exactly the same way as 
Theorem \ref{itilting}. 
\begin{proposition}\label{lem:pretiltingIU}
Let $\tilde{b}_i=b-b'$ for $b,b'\in \ol S^i_l$. 
Then $H^j(X_{\Sigma_i},\Oscr(D_{\tilde{b}_i}))=0$ for $j=1,2$.
\end{proposition}
\begin{proof}
With the aid of Proposition \ref{lem:sup:pretiltingIU} the proof follows similar lines as the proof of Lemma \ref{lem:pretilting}, therefore we only point out the necessary adjustment.  
Note that Step \ref{step1} is relevant because of \eqref{ind1}. 
 In Step \ref{step2} we need to additionally use Proposition \ref{lem:sup:pretiltingIU}, in Step \ref{step5} (4) a polygon is glued at a (contractible) string of edges. 
\end{proof}

\section{Example}
\label{sec:example}
As a sanity check we have verified our procedures for constructing
toric NCCR's for 3-dimensional toric singularities on an explicit moderately complex
example.  Let $P$ be as in Figure \ref{fig:gulotta}. We have $G=G_m^2$
and $R_P\cong (\Sym W)^G$, where $G$ acts on $W$ with weights
$(-1, 0), (3, -1),\allowbreak (-2, 3), (-2, -3), (2, 1)$ (see Figure
\ref{fig:weights}).  The Cohen-Macaulay modules of covariants
can be deduced by \cite{VdB1} or by  \cite{Bocklandt} (see Proposition
\ref{lemmaB} below).

If we follow Gulotta's (resp. the Ishii-Ueda)
procedure as shown in
Figure \ref{fig:gulotta} (resp. \ref{fig:iu}) for a sign sequence
$\mathfrak{s}=(-)_i$ we get a module of covariants $M$ corresponding to
the grey weights in the left (resp. right) Figure \ref{fig:weights}. 
We observe the following comforting facts.
\begin{itemize}
\item Both sets of grey weights do indeed contain the same number of elements (16).
\item The endomorphism rings of the corresponding modules of covariants
are Cohen-Macaulay.
\item The number 16 is also equal to $2\,\rm{Vol}(P)$, 
and so me may use Proposition \ref{prop:Vol} below to double check that the endomorphism rings are indeed NCCRs. 
\end{itemize}

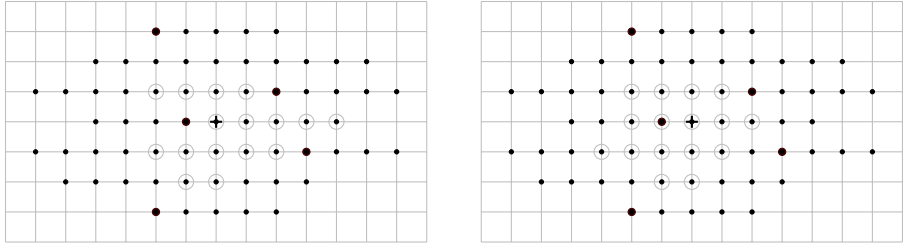
\begin{figure}[H]{}
\centering
\begin{subfigure}{.5\textwidth}
\centering
\begin{tikzpicture}[scale=0.40]
\path[draw,color=lightgray] (-7.00,-4.00) grid (7.00,4.00);
\path[draw,fill=vino,color=vino] (-1.00,0.00) circle [radius=0.12];
\path[draw,fill=vino,color=vino] (3.00,-1.00) circle [radius=0.12];
\path[draw,fill=vino,color=vino] (-2.00,3.00) circle [radius=0.12];
\path[draw,fill=vino,color=vino] (-2.00,-3.00) circle [radius=0.12];
\path[draw,fill=vino,color=vino] (2.00,1.00) circle [radius=0.12];
\path[draw,color=lightgray] (-2.00,-1.00) circle [radius=0.2500];
\path[draw,color=lightgray] (-2.00,1.00) circle [radius=0.2500];
\path[draw,color=lightgray] (-1.00,-2.00) circle [radius=0.2500];
\path[draw,color=lightgray] (-1.00,-1.00) circle [radius=0.2500];
\path[draw,color=lightgray] (-1.00,1.00) circle [radius=0.2500];
\path[draw,color=lightgray] (0.00,-2.00) circle [radius=0.2500];
\path[draw,color=lightgray] (0.00,-1.00) circle [radius=0.2500];
\path[draw,color=lightgray] (0.00,0.00) circle [radius=0.2500];
\path[draw,color=lightgray] (0.00,1.00) circle [radius=0.2500];
\path[draw,color=lightgray] (4.00,0.00) circle [radius=0.2500];
\path[draw,color=lightgray] (1.00,-1.00) circle [radius=0.2500];
\path[draw,color=lightgray] (1.00,0.00) circle [radius=0.2500];
\path[draw,color=lightgray] (1.00,1.00) circle [radius=0.2500];
\path[draw,color=lightgray] (2.00,-1.00) circle [radius=0.2500];
\path[draw,color=lightgray] (2.00,0.00) circle [radius=0.2500];
\path[draw,color=lightgray] (3.00,0.00) circle [radius=0.2500];
\path[draw,color=black,thick] (0.00,-0.200) -- (0.00,0.200);
\path[draw,color=black,thick] (-0.200,0.00) -- (0.200,0.00);
\path[draw,fill=black,color=black] (-6.00,-1.00) circle [radius=0.0700];
\path[draw,fill=black,color=black] (-6.00,1.00) circle [radius=0.0700];
\path[draw,fill=black,color=black] (-5.00,1.00) circle [radius=0.0700];
\path[draw,fill=black,color=black] (-5.00,-1.00) circle [radius=0.0700];
\path[draw,fill=black,color=black] (-5.00,-2.00) circle [radius=0.0700];
\path[draw,fill=black,color=black] (-4.00,-2.00) circle [radius=0.0700];
\path[draw,fill=black,color=black] (-4.00,-1.00) circle [radius=0.0700];
\path[draw,fill=black,color=black] (-4.00,0.00) circle [radius=0.0700];
\path[draw,fill=black,color=black] (-4.00,1.00) circle [radius=0.0700];
\path[draw,fill=black,color=black] (-4.00,2.00) circle [radius=0.0700];
\path[draw,fill=black,color=black] (-3.00,-2.00) circle [radius=0.0700];
\path[draw,fill=black,color=black] (-3.00,-1.00) circle [radius=0.0700];
\path[draw,fill=black,color=black] (-3.00,0.00) circle [radius=0.0700];
\path[draw,fill=black,color=black] (-3.00,1.00) circle [radius=0.0700];
\path[draw,fill=black,color=black] (-3.00,2.00) circle [radius=0.0700];
\path[draw,fill=black,color=black] (-2.00,-3.00) circle [radius=0.0700];
\path[draw,fill=black,color=black] (-2.00,-2.00) circle [radius=0.0700];
\path[draw,fill=black,color=black] (-2.00,-1.00) circle [radius=0.0700];
\path[draw,fill=black,color=black] (-2.00,0.00) circle [radius=0.0700];
\path[draw,fill=black,color=black] (-2.00,1.00) circle [radius=0.0700];
\path[draw,fill=black,color=black] (-2.00,2.00) circle [radius=0.0700];
\path[draw,fill=black,color=black] (-2.00,3.00) circle [radius=0.0700];
\path[draw,fill=black,color=black] (-1.00,-3.00) circle [radius=0.0700];
\path[draw,fill=black,color=black] (-1.00,-2.00) circle [radius=0.0700];
\path[draw,fill=black,color=black] (-1.00,-1.00) circle [radius=0.0700];
\path[draw,fill=black,color=black] (-1.00,0.00) circle [radius=0.0700];
\path[draw,fill=black,color=black] (-1.00,1.00) circle [radius=0.0700];
\path[draw,fill=black,color=black] (-1.00,2.00) circle [radius=0.0700];
\path[draw,fill=black,color=black] (-1.00,3.00) circle [radius=0.0700];
\path[draw,fill=black,color=black] (0.00,-3.00) circle [radius=0.0700];
\path[draw,fill=black,color=black] (0.00,-2.00) circle [radius=0.0700];
\path[draw,fill=black,color=black] (0.00,-1.00) circle [radius=0.0700];
\path[draw,fill=black,color=black] (0.00,0.00) circle [radius=0.0700];
\path[draw,fill=black,color=black] (0.00,1.00) circle [radius=0.0700];
\path[draw,fill=black,color=black] (0.00,2.00) circle [radius=0.0700];
\path[draw,fill=black,color=black] (0.00,3.00) circle [radius=0.0700];
\path[draw,fill=black,color=black] (1.00,-3.00) circle [radius=0.0700];
\path[draw,fill=black,color=black] (1.00,-2.00) circle [radius=0.0700];
\path[draw,fill=black,color=black] (1.00,-1.00) circle [radius=0.0700];
\path[draw,fill=black,color=black] (1.00,0.00) circle [radius=0.0700];
\path[draw,fill=black,color=black] (1.00,1.00) circle [radius=0.0700];
\path[draw,fill=black,color=black] (1.00,2.00) circle [radius=0.0700];
\path[draw,fill=black,color=black] (1.00,3.00) circle [radius=0.0700];
\path[draw,fill=black,color=black] (2.00,-3.00) circle [radius=0.0700];
\path[draw,fill=black,color=black] (2.00,-2.00) circle [radius=0.0700];
\path[draw,fill=black,color=black] (2.00,-1.00) circle [radius=0.0700];
\path[draw,fill=black,color=black] (2.00,0.00) circle [radius=0.0700];
\path[draw,fill=black,color=black] (2.00,1.00) circle [radius=0.0700];
\path[draw,fill=black,color=black] (2.00,2.00) circle [radius=0.0700];
\path[draw,fill=black,color=black] (2.00,3.00) circle [radius=0.0700];
\path[draw,fill=black,color=black] (4.00,-1.00) circle [radius=0.0700];
\path[draw,fill=black,color=black] (4.00,0.00) circle [radius=0.0700];
\path[draw,fill=black,color=black] (4.00,1.00) circle [radius=0.0700];
\path[draw,fill=black,color=black] (4.00,2.00) circle [radius=0.0700];
\path[draw,fill=black,color=black] (3.00,-2.00) circle [radius=0.0700];
\path[draw,fill=black,color=black] (3.00,-1.00) circle [radius=0.0700];
\path[draw,fill=black,color=black] (3.00,0.00) circle [radius=0.0700];
\path[draw,fill=black,color=black] (3.00,1.00) circle [radius=0.0700];
\path[draw,fill=black,color=black] (3.00,2.00) circle [radius=0.0700];
\path[draw,fill=black,color=black] (5.00,-1.00) circle [radius=0.0700];
\path[draw,fill=black,color=black] (5.00,1.00) circle [radius=0.0700];
\path[draw,fill=black,color=black] (5.00,2.00) circle [radius=0.0700];
\path[draw,fill=black,color=black] (6.00,-1.00) circle [radius=0.0700];
\path[draw,fill=black,color=black] (6.00,1.00) circle [radius=0.0700];
\end{tikzpicture}
\end{subfigure}%
\begin{subfigure}{.5\textwidth}
\centering
\begin{tikzpicture}[scale=0.40]
\path[draw,color=lightgray] (-7.00,-4.00) grid (7.00,4.00);
\path[draw,fill=vino,color=vino] (-1.00,0.00) circle [radius=0.12];
\path[draw,fill=vino,color=vino] (3.00,-1.00) circle [radius=0.12];
\path[draw,fill=vino,color=vino] (-2.00,3.00) circle [radius=0.12];
\path[draw,fill=vino,color=vino] (-2.00,-3.00) circle [radius=0.12];
\path[draw,fill=vino,color=vino] (2.00,1.00) circle [radius=0.12];
\path[draw,color=lightgray] (-3.00,-1.00) circle [radius=0.250];
\path[draw,color=lightgray] (-2.00,-1.00) circle [radius=0.250];
\path[draw,color=lightgray] (-2.00,0.00) circle [radius=0.250];
\path[draw,color=lightgray] (-2.00,1.00) circle [radius=0.250];
\path[draw,color=lightgray] (-1.00,-2.00) circle [radius=0.250];
\path[draw,color=lightgray] (-1.00,-1.00) circle [radius=0.250];
\path[draw,color=lightgray] (-1.00,0.00) circle [radius=0.250];
\path[draw,color=lightgray] (-1.00,1.00) circle [radius=0.250];
\path[draw,color=lightgray] (0.00,-2.00) circle [radius=0.250];
\path[draw,color=lightgray] (0.00,-1.00) circle [radius=0.250];
\path[draw,color=lightgray] (0.00,0.00) circle [radius=0.250];
\path[draw,color=lightgray] (0.00,1.00) circle [radius=0.250];
\path[draw,color=lightgray] (1.00,-1.00) circle [radius=0.250];
\path[draw,color=lightgray] (1.00,0.00) circle [radius=0.250];
\path[draw,color=lightgray] (1.00,1.00) circle [radius=0.250];
\path[draw,color=lightgray] (2.00,0.00) circle [radius=0.250];
\path[draw,color=black,thick] (0.00,-0.200) -- (0.00,0.200);
\path[draw,color=black,thick] (-0.200,0.00) -- (0.200,0.00);
\path[draw,fill=black,color=black] (-6.00,-1.00) circle [radius=0.0700];
\path[draw,fill=black,color=black] (-6.00,1.00) circle [radius=0.0700];
\path[draw,fill=black,color=black] (-5.00,1.00) circle [radius=0.0700];
\path[draw,fill=black,color=black] (-5.00,-1.00) circle [radius=0.0700];
\path[draw,fill=black,color=black] (-5.00,-2.00) circle [radius=0.0700];
\path[draw,fill=black,color=black] (-4.00,-2.00) circle [radius=0.0700];
\path[draw,fill=black,color=black] (-4.00,-1.00) circle [radius=0.0700];
\path[draw,fill=black,color=black] (-4.00,0.00) circle [radius=0.0700];
\path[draw,fill=black,color=black] (-4.00,1.00) circle [radius=0.0700];
\path[draw,fill=black,color=black] (-4.00,2.00) circle [radius=0.0700];
\path[draw,fill=black,color=black] (-3.00,-2.00) circle [radius=0.0700];
\path[draw,fill=black,color=black] (-3.00,-1.00) circle [radius=0.0700];
\path[draw,fill=black,color=black] (-3.00,0.00) circle [radius=0.0700];
\path[draw,fill=black,color=black] (-3.00,1.00) circle [radius=0.0700];
\path[draw,fill=black,color=black] (-3.00,2.00) circle [radius=0.0700];
\path[draw,fill=black,color=black] (-2.00,-3.00) circle [radius=0.0700];
\path[draw,fill=black,color=black] (-2.00,-2.00) circle [radius=0.0700];
\path[draw,fill=black,color=black] (-2.00,-1.00) circle [radius=0.0700];
\path[draw,fill=black,color=black] (-2.00,0.00) circle [radius=0.0700];
\path[draw,fill=black,color=black] (-2.00,1.00) circle [radius=0.0700];
\path[draw,fill=black,color=black] (-2.00,2.00) circle [radius=0.0700];
\path[draw,fill=black,color=black] (-2.00,3.00) circle [radius=0.0700];
\path[draw,fill=black,color=black] (-1.00,-3.00) circle [radius=0.0700];
\path[draw,fill=black,color=black] (-1.00,-2.00) circle [radius=0.0700];
\path[draw,fill=black,color=black] (-1.00,-1.00) circle [radius=0.0700];
\path[draw,fill=black,color=black] (-1.00,0.00) circle [radius=0.0700];
\path[draw,fill=black,color=black] (-1.00,1.00) circle [radius=0.0700];
\path[draw,fill=black,color=black] (-1.00,2.00) circle [radius=0.0700];
\path[draw,fill=black,color=black] (-1.00,3.00) circle [radius=0.0700];
\path[draw,fill=black,color=black] (0.00,-3.00) circle [radius=0.0700];
\path[draw,fill=black,color=black] (0.00,-2.00) circle [radius=0.0700];
\path[draw,fill=black,color=black] (0.00,-1.00) circle [radius=0.0700];
\path[draw,fill=black,color=black] (0.00,0.00) circle [radius=0.0700];
\path[draw,fill=black,color=black] (0.00,1.00) circle [radius=0.0700];
\path[draw,fill=black,color=black] (0.00,2.00) circle [radius=0.0700];
\path[draw,fill=black,color=black] (0.00,3.00) circle [radius=0.0700];
\path[draw,fill=black,color=black] (1.00,-3.00) circle [radius=0.0700];
\path[draw,fill=black,color=black] (1.00,-2.00) circle [radius=0.0700];
\path[draw,fill=black,color=black] (1.00,-1.00) circle [radius=0.0700];
\path[draw,fill=black,color=black] (1.00,0.00) circle [radius=0.0700];
\path[draw,fill=black,color=black] (1.00,1.00) circle [radius=0.0700];
\path[draw,fill=black,color=black] (1.00,2.00) circle [radius=0.0700];
\path[draw,fill=black,color=black] (1.00,3.00) circle [radius=0.0700];
\path[draw,fill=black,color=black] (2.00,-3.00) circle [radius=0.0700];
\path[draw,fill=black,color=black] (2.00,-2.00) circle [radius=0.0700];
\path[draw,fill=black,color=black] (2.00,-1.00) circle [radius=0.0700];
\path[draw,fill=black,color=black] (2.00,0.00) circle [radius=0.0700];
\path[draw,fill=black,color=black] (2.00,1.00) circle [radius=0.0700];
\path[draw,fill=black,color=black] (2.00,2.00) circle [radius=0.0700];
\path[draw,fill=black,color=black] (2.00,3.00) circle [radius=0.0700];
\path[draw,fill=black,color=black] (4.00,-1.00) circle [radius=0.0700];
\path[draw,fill=black,color=black] (4.00,0.00) circle [radius=0.0700];
\path[draw,fill=black,color=black] (4.00,1.00) circle [radius=0.0700];
\path[draw,fill=black,color=black] (4.00,2.00) circle [radius=0.0700];
\path[draw,fill=black,color=black] (3.00,-2.00) circle [radius=0.0700];
\path[draw,fill=black,color=black] (3.00,-1.00) circle [radius=0.0700];
\path[draw,fill=black,color=black] (3.00,0.00) circle [radius=0.0700];
\path[draw,fill=black,color=black] (3.00,1.00) circle [radius=0.0700];
\path[draw,fill=black,color=black] (3.00,2.00) circle [radius=0.0700];
\path[draw,fill=black,color=black] (5.00,-1.00) circle [radius=0.0700];
\path[draw,fill=black,color=black] (5.00,1.00) circle [radius=0.0700];
\path[draw,fill=black,color=black] (5.00,2.00) circle [radius=0.0700];
\path[draw,fill=black,color=black] (6.00,-1.00) circle [radius=0.0700];
\path[draw,fill=black,color=black] (6.00,1.00) circle [radius=0.0700];
\end{tikzpicture}
\end{subfigure}
\caption{
\small\textcolor{vino}{${\bullet}$} $G$-weights, \textcolor{black}{\tiny$\bullet$} CM weights, \textcolor{lightgray}{$\bigcirc$} Gulotta (left), Ishii-Ueda (right)}
\label{fig:weights}
\end{figure}

\appendix{}

\appendix{}
\section{Toric DM stacks and semi-orthogonal decompositions}\label{stacks}
As usual we assume $M=N=\ZZ^n$. In this appendix we show that the derived category of a smooth toric 
DM stack $X_{\bf\Sigma}$ (see \cite{BorisovHorja}) for
${\bf\Sigma}=(\Sigma,(n_i)_{i=1}^\kl)$, where $|\Sigma|$ is a
(rational strongly convex) polyhedral cone over a convex lattice
polyhedron $P\times \{1\}\subset \RR^n$ and $n_i\in P\times \{1\}$ are
generators of $1$-dimensional cones in $\Sigma$, does not admit any
nontrivial semi-orthogonal decomposition.

We first recall a result from \cite{borisov2005orbifold}. Here  $\Sigma_{\max}$ denotes the set of maximal cones in $\Sigma$. 
\begin{proposition}\cite[Proposition 4.3 and its proof]{borisov2005orbifold}\label{prop:ocovst}
Let $\sigma$ be a maximal cone in  $\Sigma$.
  If $\sigma=[n_{i_1},\dots,n_{i_n}]$, then 
$X_{\boldsymbol\sigma}$, ${\boldsymbol\sigma}=(\sigma,(n_{i_j})_{j=1}^n)$, corresponds to an open substack of $X_{\bf\Sigma}$ and $X_{\bf\Sigma}=\bigcup_{\sigma\in\Sigma_{\max}}X_{\bf\sigma}$. 
Moreover $X_{\boldsymbol\sigma}\cong k^n/ G_\sigma$ for a finite subgroup $G_\sigma\subset G$ (with $G$ is as in  \S\ref{sec:prel}) such that the action of $G_\sigma$ on $k^n$ is unimodular and faithful. 
\end{proposition}

{Recall that we denote $X_P=\Spec(R_{\sigma_P})$.}
\begin{lemma}\label{rmk:crepant}
Let $\pi_s:X_{\bf\Sigma}\to X_{\Sigma}$ be the canonical map. 
Then 
\[\Oscr_{X_{\bf\Sigma}}\cong\omega_{X_{\bf\Sigma}}\cong \pi_s^*\omega_{X_{\Sigma}}\cong \pi_s^*\Oscr_{X_{\Sigma}}.
\]
\end{lemma}

\begin{proof}
Let $\sigma\in \Sigma_{\max}$. By Proposition \ref{prop:ocovst} the  smooth points in $X_{\sigma}$ correspond exactly to $G_\sigma$-orbits of points in $k^n$ with trivial stabilizer. 
Indeed, since $G_\sigma$ is unimodular
the stabilizer groups cannot contain pseudo-reflections, and therefore the corresponding quotient (of the $G_\sigma$-orbit of a point in $k^n$ with nontrivial stabilizer) cannot be regular by the Chevalley-Shephard-Todd theorem. 

Thus, as both  $\omega_{X_{\bf \Sigma}}$ and $\omega_{X_\Sigma}$ are reflexive and $\pi_s:X_{\bf \Sigma}\to X_\Sigma$ is the identity on the smooth locus of $X_\Sigma$ 
(whose complement has codimension $\geq 2$ as $X_\Sigma$ is a normal variety) we have $\pi_s^*\omega_{X_\Sigma}=\omega_{X_{\bf \Sigma}}$.   
Moreover, $\tau:X_\Sigma \to X_P$ is crepant; i.e., 
$\tau^* \omega_{X_P}=\omega_{X_\Sigma}$
 (see e.g. \cite[Example 11.2.6]{CoxLittleSchenck}) and as $X_P$ is Gorenstein it follows
  $\omega_{X_{\bf\Sigma}}\cong \Oscr_{X_{\bf \Sigma}}$.
\end{proof}

\begin{lemma}\label{relSerreaff}
The relative Serre functor\footnote{See \cite{Kuz}.} of $\Dscr(X_{\bf \Sigma})$ with respect to $\Dscr(X_P)$ is the identity; i.e. for all $\Fscr,\Gscr\in \Dscr(X_{\bf \Sigma})$ we have 
\begin{equation}\label{eq:relSerre}
\RHom_{X_{\bf \Sigma}}(\Fscr,\Gscr)\cong \RHom_{R_P}(\RHom_{X_{\bf \Sigma}}(\Gscr,\Fscr),R_P).
\end{equation}
\end{lemma}

\begin{proof}
Since $\pi_s^*\omega_{X_{\Sigma}}=\omega_{X_{\bf\Sigma}}$ by Lemma \ref{rmk:crepant}, it follows from  \cite[Theorem 3.0.17]{Abuaf2} and its proof that the relative Serre functor of $\Dscr(X_{\bf \Sigma})$ with respect to $\Dscr(X_\Sigma)$ is the identity. It is easy to check that it is thus enough to show that the identity is a relative Serre functor of $\Dscr(X_\Sigma)$ with respect to $\Dscr(X_P)$. Since $\theta:X_\Sigma\to X_P$ is crepant (see proof of Lemma \ref{rmk:crepant}) and $X_\Sigma$, $X_P$ are Gorenstein, this follows by the same proof as in \cite[Lemma 4.13]{IW1}.
\end{proof}

\begin{lemma}\label{lem:connected}
$\Dscr(X_{\bf\Sigma})$ is indecomposable.
\end{lemma}

\begin{proof}
Let $\sigma$ be a maximal cone in $\Sigma$. 
Then $X_{\boldsymbol\sigma}\cong k^n/G_\sigma$ for a finite group $G_\sigma$, which acts faithfully on $k^n$ by Proposition  \ref{prop:ocovst}. 
Thus, $\Dscr(X_{\boldsymbol\sigma})$ is indecomposable by \cite[Lemma 4.2]{BKR}. 
Assume that we have an orthogonal decomposition
$\Dscr=\Dscr(X_{\bf\Sigma})=\la \Dscr_1,\Dscr_2\ra$, $\Hom(\Dscr_1,\Dscr_2)=\Hom(\Dscr_2,\Dscr_1)=0$. 
Since $\Dscr(X_{\bf\sigma})$ is a Verdier quotient of $\Dscr(X_{\bf\Sigma})$ (as $X_{\bf \Sigma}$ is a noetherian stack by \cite[Prop. 15.4]{champs}) 
 and if $C=C_1\oplus C_2$, $C_i\in \Dscr_i$, has support in $X_{\bf\Sigma}\setminus U_{\bf\Sigma}$, the same is true for $C_1,C_2$, the orthogonal decomposition is preserved when passing to an open substack. 
For  $\sigma\in \Sigma_{\max}$  
we thus have $j_\sigma^*\Dscr_1=0$ or $j_\sigma^*\Dscr_2=0$ (where $j_\sigma:X_{\boldsymbol\sigma}\hookrightarrow X_{\bf\Sigma}$). 
Choose a point $x$ in $X_\Sigma$ which is in the intersection of the $X_\sigma$ and also in the locus where $\pi_s$ is an isomorphism. Then we can view $\Oscr_x$ as an object in $\Dscr$.
Since $\Oscr_x$ is indecomposable we can assume that $\Oscr_x\in \Dscr_1$. Since the restriction of $\Oscr_x$ to $X_{\boldsymbol\sigma}$ is nonzero for every $\sigma\in \Sigma_{\max}$, we get that $j_\sigma^*\Dscr_2=0$ for every $\sigma\in \Sigma_{\max}$, which easily implies that 
$\Dscr_2=0$.
\end{proof}

\begin{corollary}\label{cor:sod}
$\Dscr(X_{\bf \Sigma})$ does not admit any nontrivial semi-orthogonal decompositions. 
\end{corollary}{}

\begin{proof}
It follows by \eqref{eq:relSerre} that $\RHom_{X_{\bf \Sigma}}(\Gscr,\Fscr)=0$ implies  $\RHom_{X_{\bf \Sigma}}(\Fscr,\Gscr)=0$. Thus, if $\Dscr(X_{\bf \Sigma})$ has a non-trivial semi-orthogonal decomposition then it is decomposable. This contradicts Lemma \ref{lem:connected}.
\end{proof}

\section{Cohen-Macaulay modules}\label{appA}
In this section we show how one can extend a useful criterion of Bocklandt \cite{Bocklandt} for recognising Cohen-Macaulay modules of covariants in dimension $3$.   
It can be used to give a direct proof of the Cohen-Macaulayness of the endomorphism algebra $\Lambda=\End_{R_P}(\bigoplus_{b\in \ol S^i_l}\M_b)$ in \S\ref{Gulotta}. Moreover, we will explain that Theorem \ref{Broom} could be deduced from this fact provided we knew that $|\ol S^i_l|=2\Vol(P)$. 

Recall the following result.\footnote{The result may also be deduced from \eqref{eq:cohomMb} by  calculating the cohomology of $M_b$ restricted to the punctured spectrum of $\Spec R_P$. The latter is the toric variety associated to the fan spanned by $\partial P$.}
\begin{proposition}\cite[Lemma 4.5]{Bocklandt}\label{lemmaB}
Let $\sigma=[n_1,\dots,n_k]$, with $n_1,\dots,n_k$ oriented counter-clockwise, be a minimal presentation of $\sigma$, and let 
$b\in \ZZ^k$.  
The $R_P$-module $\M_b$ is Cohen-Macaulay  if and only 
  for every $m\in M$ the signs $s_{i}^b(m)$ follow the 
pattern $+\cdots+-\cdots-$ (up to cyclic permutation and possibly with no $+$ or $-$ present).
\end{proposition}

Let $P\subset \RR^2\times{1}$ be a lattice polygon with vertices $n_1,\ldots,n_k$ and  let $P'$ denote a lattice polygon obtained from $P$ by removing a triangle with vertex $n_k$.  
 We denote the other vertices of the triangle by $n_{k+1}$, $n_{k+2}$. They lie on the boundary of $P$ (and may, or may not, be vertices). 
Then $[n_1,\dots,n_k,n_{k+1},n_{k+2}]$ is convexly induced from $[n_1,\dots,n_k]$ with an induction datum $\mathfrak{t}=(t^{(k+1)},t^{(k+2)})$, where $t^{(k+1)}_j=0$ for $j\neq k-1,k$, $t^{(k+2)}_j=0$ for $j\neq 1,k$. For $b\in \ZZ^k$ and $\mathfrak{s}\in \{+,-\}^2$ we let $b_{\mathfrak{s}}$ be convexly induced from $b$ with the sign sequence $\mathfrak{s}$ (and the induction datum $\mathfrak{t}$). By $\ol{b_{\mathfrak{s}}}$ we denote the projection of $b_{\mathfrak{s}}$ to 
  $\ZZ^{k+1}$ omitting the $k$-th component.

Using the above criterion one can
use a simplified version of Step \ref{step2} in the proof of Lemma \ref{lem:pretilting} (with $\{w_1,w_2,\ldots\}=\{n_1,\ldots,n_{k+2}\}$)
 to obtain  Corollary \ref{pmcm} below. This is a generalization of \cite[Corollary 4.7]{Bocklandt} relaxing the assumption that  
 $\{n_{k+1},n_{k+2}\}\subset \{n_1,\ldots,n_{k-1}\}$
(in the latter case $\ol{b_{\mathfrak{s}}}$ is just the projection of $b$ and in particular the sign sequence $\mathfrak{s}$ is irrelevant).
\begin{corollary}\label{pmcm} 
If $\M_b$ is a Cohen-Macaulay $R_P$-module 
 then $\M_{\ol{b_{\mathfrak{s}}}}$ is a Cohen-Macaulay $R_{\nP}$-module. 
Moreover, if $\End_{R_P} (\oplus_{b\in S} \M_b)$ is a Cohen-Macaulay $R_P$-module  then the $R_{\nP}$-module $\End_{R_{\nP}}(\oplus_{b\in S}\M_{\ol{b_{\mathfrak{s}}}})$ is also Cohen-Macaulay. 
\end{corollary}

If we could prove that using the Gulotta ``cutting'' procedure combined with Corollary \ref{pmcm} one gets a Cohen Macaulay $R_P$-module $M$ such that $\End_{R_P}(M)$ is Cohen-Macaulay and such 
that $M$ has precisely  $2\Vol(P)$ non-iso\-morphic indecomposable summands then one obtains by the following proposition (see also \cite[proof of Theorem 7.1]{Bocklandt}) a proof of Theorem \ref{Broom} which does not require the construction of a tilting bundle on a resolution.
 Unfortunately we have not been able to solve this purely combinatorial problem.

\begin{proposition}\label{prop:Vol}
Let $M$ be a graded $R_P$-module (for some connected grading on $R_P$)  such that $M$ contains at least $2\Vol(P)$ non-isomorphic graded  indecomposable summands and let $\Lambda=\End_{R_P}(M)$.
If $\Lambda$ is a Cohen-Macaulay $R_P$-module, then $\Lambda$ is an NCCR of $R_P$. 
\end{proposition}

\begin{proof}
Since $R_P$ has an NCCR $A$ by \cite[Proposition 1.1, Corollary 3.2]{SVdB4}, it has a CCR which is derived equivalent to $A$ by \cite[Theorem 6.3.1]{VdB32}. 
By \cite{Bridgeland}, all crepant resolutions of $R_P$ are derived equivalent. There is a crepant resolution of $R_P$ given by $X_\Sigma$ corresponding to a triangulation of $P$ with triangles of area 
$1/2$.
By \cite{BorisovHorja} (see \cite[Theorem A.1]{SVdB4}), $\rank K(X_\Sigma)=2\Vol(P)$, 
thus $\rank K_0(A)=2\Vol(P)$. 

Since $M$ is a modifying module, it is a direct summand of a graded\footnote{The results of \cite{IW} are valid in the graded context.} ``maximal
modifying module'' $M\oplus M'$ by \cite[Proposition 1.19]{IW} and we may assume that $M$ and $M'$ have no common indecomposable summands.  
Since $R_P$ has
an NCCR, the ``maximal modification algebra'' $\Lambda':=\End_{R_P}(M\oplus M')$ is also an NCCR by \cite[Proposition 1.13]{IW}. As all
NCCRs of $R_P$ are derived equivalent by \cite[Theorem
1.5]{IyamaWemyssNCBO}, $\rank K_0(\Lambda')=\rank K_0(A)=2\Vol(P)$. Let $m$, $m'$ be the number of non-isomorphic indecomposable summands of $M$, $M'$. By Lemma \ref{lem:summands} 
below applied to $\Lambda'$ we now have $2\Vol(P)=\rank K_0(\Lambda')\ge m+m'\ge m\ge 2\Vol(P)$ (the last inequality by the hypotheses) and hence $m'=0$. Therefore $M'=0$ and $\Lambda'=\Lambda$. Hence
$\Lambda$ is an NCCR of $R_P$.
\end{proof}
We have used the following lemma.
\begin{lemma} \label{lem:summands}
Let $R$ be a commutative finitely generated connected
  graded ring and let $M$ be be a finitely generated graded $R$-module with
  $m$ non-isomorphic indecomposable summands. Put $
  \Lambda=\End_R(M)$. Then $\rank K_0(\Lambda)\ge m$.
\end{lemma}
\begin{proof} Let $\rad \Lambda$ be the graded radical of $\Lambda$. Then the morphism between Grothen\-dieck groups $K_0(\Lambda)\r K_0(\Lambda/\rad \Lambda)$
is surjective (by lifting of idempotents) and moreover $\rank K_0(\Lambda/\rad\Lambda)=m$, finishing the proof.
\end{proof}

\section{Tilting bundles vs NCCRs}\label{appC}
In dimension $3$ (split) tilting bundles on crepant resolutions and (toric) NCCRs are
intimately connected. The following theorem summarizes the relation.
While the first part (see \cite{VdB32}) of the theorem is standard, the converse (see \cite{IUconj}) is quite involved and requires a deep
study of the delicate interplay between crepant resolutions, dimer
models, and moduli spaces of representations of the corresponding quivers,
and moreover NCCRs (see \cite{Bocklandt}).  Since such moduli spaces  are projective over affine, the converse in addition requires that the resolution is projective. 
\begin{theorem}\cite{VdB32}, \cite[Corollary 1.2]{IUconj}, \cite[Theorem 7.1.-7.3]{Bocklandt} \label{th:tilting}
Let $R$ be the coordinate ring of a 3-dimensional Gorenstein affine toric variety
and let $\pi:X\to \Spec  R$ be a crepant resolution.  If  $\Tscr$ is a tilting bundle on $X$ then $\End_X(\Tscr)$ is an NCCR of $R$. 
Moreover, if $\pi$ is projective and $\Lambda$ is a toric NCCR of $R$, then  there exists a split tilting bundle $\Tscr$ on $X$ such that $\Lambda\cong \End_X(\Tscr)$. 
\end{theorem}
Currently the full generality of this result seems out of
reach of the methods used in this paper. This being said, the convex
induction approach makes it often possible to find the tilting bundles whose existence is asserted in Theorem \ref{th:tilting} in
a rather direct way. Interestingly, as we show below, convex induction is sometimes able to provide
tilting bundles even if the resolution is not projective. This situation is not covered by Theorem \ref{th:tilting}.

We list some examples where the convex induction method yields split tilting bundles on a  crepant resolution of an affine toric variety for any toric NCCR. Note that any such crepant resolution 
is toric 
and 
is obtained from the fan  associated to a lattice triangulation with triangles of area 1/2 of the polygon $P$ corresponding to $R$ (see \cite[Proposition 2.4]{MR1808643}, such a triangulation is called ``unimodular'').
\begin{enumerate}
\item\label{addA}
One way to triangulate $P$ is by using the Ishii-Ueda method (successively removing corner vertices and taking the convex hull of the remaining lattice points). 
For this type of triangulation we can  induce the $b$'s defined on the vertices of $P$, corresponding to the toric NCCR, to the vertices of the triangulation using triangle convex induction
as in \S\ref{IshiiUeda},  and define the corresponding split tilting bundle.

\item
Convex induction may still work even if the triangulation is not as in \eqref{addA}. For example, the star triangulation of a (minimal) regular lattice hexagon (see Figure \ref{fig:hexagon}) 
 is not of the above form, but one may use  convex induction to find a split tilting bundle. 
 Let $\Lambda=\End_R(\oplus_{b\in S}M_b)$ be a toric NCCR. Let $n_1,\dots,n_6$ be vertices of the polygon and let $n_7$ be the middle point. We  fix some induction datum $(t^{(7)}_{j})_j$ 
(e.g. $t^{(7)}_j=1/6$ for $j=1,\ldots,6$) and  let $\tilde{S}$ be compatibly convexly induced from $S$. We claim that  $\oplus_{\tilde{b}\in\tilde{S}} \Oscr(D_{\tilde{b}})$ is a tilting bundle on $X$. We need to show (see \eqref{eq:Mbdef}, \eqref{vdbm}) that 
 $\tilde{H}^{\geq 0}(V_{D_{\tilde{b}-\tilde{b}'},m},k)=0$ for $\tilde{b},\tilde{b}'\in \tilde{S}$ and $m\in M$. 
 It follows from the fact that $\Lambda$ is Cohen-Macaulay (see Proposition \ref{lemmaB}) that any sign sequence on the boundary of the polygon will be of the form $+,\dots,+,-,\dots,-$ (with possibly no $+$ or $-$ appearing), therefore $V_{D_{\tilde b-\tilde{b}'},m}$ would not be empty or contractible only if the sign pattern on the boundary is $-,\dots,-$ and there is $+$ in the middle. However, this is impossible by the convex induction (using Lemma \ref{diff}).
\end{enumerate}

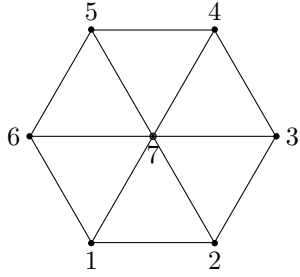
\begin{figure}[H]{}
\begin{tikzpicture}[scale=0.50]
\path
    node[
      regular polygon,
      regular polygon sides=6,
      draw,
      inner sep=1cm,
    ] (hexagon) {}
    %
    (hexagon.corner 1) node[above] {$4$}
    (hexagon.corner 2) node[above] {$5$}
    (hexagon.corner 3) node[left] {$6$}
    (hexagon.corner 4) node[below] {$1$}
    (hexagon.corner 5) node[below] {$2$}
    (hexagon.corner 6) node[right] {$3$}
    (hexagon.center) node[below] {$7$}
    %
    plot[
      mark=*,
      samples at={1, ..., 6},
    ] (hexagon.corner \x);
\tkzDrawPoints(hexagon.center);
\path[draw,](hexagon.corner 1)--(hexagon.corner 4);
\path[draw,](hexagon.corner 2)--(hexagon.corner 5);
\path[draw,](hexagon.corner 3)--(hexagon.corner 6);
\end{tikzpicture}
\caption{A crepant resolution of a hexagon.}
  \label{fig:hexagon}
\end{figure}
The combinatorial argument used in the last example can be used in more
general situations via the following lemma.
\begin{lemma} \label{lem:abstraction}
Let $(n_i)_{i\in I}$ be the vertices of a unimodular triangulation of~$P$ and assume that the $(n_i)_i$ are obtained by convex induction from the vertices of $P$
for a chosen induction datum. Assume that
every set of signs $(s_i)_{i\in I}\in \{\pm \}^I$ compatible with Lemma \ref{lem:sign} whose the restriction to the vertices of $P$ 
is of the form $+,\ldots,+,-,\ldots,-$ (possibly reflected and with at least one $+$ and one $-$),
has the property
that every $n_i$ is connected with a constant sign path to $\partial P$. Then $R$ satisfies the converse
of Theorem \ref{th:tilting} for the (not necessarily projective) crepant resolution of $\Spec R$
given by the triangulation. 
\end{lemma}
\begin{proof}
Let $(n_j)_{j\in J}$ be the vertices of $P$ and let 
let $S=\{(b_j)_{j\in J}\}$
be a collection of $b$'s defining an NCCR of $R$ and let $\tilde{S}=\{(\tilde{b}_i)_{i\in I}\}$ be obtained from
$S$ by compatible convex induction for an arbitrary sign sequence and the given induction datum. Let $\tilde{b}=\tilde{b}'-\tilde{b}''$
for some $\tilde{b}',\tilde{b}''\in \tilde{S}$. Then $\tilde{b}$ is obtained by convex induction
from $b:=b'-b''$ by Lemma \ref{diff}.

For $m\in M$ put $s_i=s_{i}^{\tilde{b}}(m)$. Then $(s_i)_{i\in I}$ is compatible with Lemma \ref{lem:sign}.
We will show that $V_{D_{\tilde{b}},m}$ is either empty or contractible.

Since $S$ defines an NCCR the restriction of $(s_i)_{i\in I}$ to the vertices  
 of $P$ will be of the form $+,\dots,+,-,\dots,-$ (with possibly no $+$ or $-$ appearing)
(see Proposition \ref{lemmaB}). By Lemma \ref{lem:sign} the same will
be true for the restriction to the vertices of the triangulation in $\partial P$.

If the sequence restricted to $\partial P$ is $+,\ldots,+$ then by Lemma \ref{lem:sign}: $V_{\tilde{b},m}=\emptyset$. If the
sequence is $-,\ldots,-$ then $V_{D_{\tilde{b}},m}=P$.

Assume now that the restricted sequence is $+,\ldots,+,-,\ldots,-$. Since every $n_i$ has a constant sign path
to the boundary the set of $-$'s is connected and moreover no cycle of $-$'s can surround a $+$. It now suffices
to apply Lemma \ref{lem:zastrow} below.
\end{proof}
\begin{lemma} 
\label{lem:zastrow} Assume we are given a triangulation of $P$ with vertices
  $(n_i)_{i\in I}$. Let $b\in \ZZ^I$, $m\in M$ and assume in addition
  that $(s^b_i(m))_i$ is such that the set of $-$'s is connected and
  furthermore that no cycle of $-$'s surrounds a $+$. Then $V_{D_b,m}$ is
  either empty or contractible.
\end{lemma}
\begin{proof} Put $V=V_{D_b,m}$.
Assume $V\neq \emptyset$. Then the hypotheses imply $|\pi_i V|=1$ for $i=0,1$.
 Furthermore as $V$ is planar, Zastrow's theorem \cite[Theorem A]{Zastrow} asserts that $|\pi_iV|=1$ for $i\ge 2$.
Since $V$ is a CW-complex this implies that $V$ is contractible by Whitehead's theorem.
\end{proof}
\begin{enumerate}
  \setcounter{enumi}{2}
\item
Projective crepant resolutions correspond to regular triangulation of $P$ (see e.g. \cite[Definition 15.2.8]{CoxLittleSchenck}). However, there are also non-regular triangulations for which the convex induction method works. As an example (see Figure \ref{fig:nonregular}) we take the triangulation \cite[\S 5.2]{Logvinenko} refining the iconic  (minimal) example of a non-regular triangulation \cite[Example 1.4]{Santos}. 
\begin{figure}[H]{}
\begin{tikzpicture}[scale=0.50]
\tkzDefPoint(0,0){1}
\tkzDefPoint(8,0){2}
\tkzDefMidPoint(1,2)\tkzGetPoint{4}
\tkzDefTriangle[equilateral](1,2)\tkzGetPoint{3}
\tkzDefMidPoint(2,3)\tkzGetPoint{5}
\tkzDefMidPoint(3,1)\tkzGetPoint{6}
\tkzInterLL(1,5)(2,6)\tkzGetPoint{7}
\tkzDefMidPoint(1,7)\tkzGetPoint{8}
\tkzDefMidPoint(2,7)\tkzGetPoint{9}
\tkzDefMidPoint(3,7)\tkzGetPoint{10}
\tkzDrawPolygon(1,2,3)
\tkzDrawSegments(1,8)
\tkzDrawSegments(1,9)
\tkzDrawSegments(4,9)
\tkzDrawSegments(9,2)
\tkzDrawSegments(2,10)
\tkzDrawSegments(5,10)
\tkzDrawSegments(3,10)
\tkzDrawSegments(3,8)
\tkzDrawSegments(6,8)
\tkzDrawSegments(8,9)
\tkzDrawSegments(10,9)
\tkzDrawSegments(10,8)
\tkzDrawSegments(7,8)
\tkzDrawSegments(7,9)
\tkzDrawSegments(10,7)
\tkzDrawPoints(1,2,3,4,5,6,7,8,9,10)

\node[] at (0,-0.4) {$1$};

\node[] at (4,-0.4) {$4$};

\node[] at (8,-0.4) {$2$};

\node[] at (4,7.3) {$3$};

\node[] at (6.3,3.6) {$5$};

\node[] at (1.7,3.6) {$6$};

\node[] at (4,1.9) {$7$};

\node[] at (2,0.8) {$8$};

\node[] at (6,0.8) {$9$};

\node[] at (4,5) {$10$};

\end{tikzpicture}
\caption{A non-regular triangulation}
\label{fig:nonregular}
\end{figure}
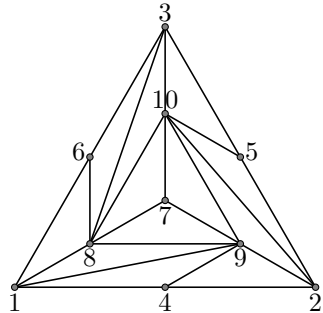

We induce $[n_1,n_2,n_3,n_4,\dots,n_{10}]$ from $[n_1,n_2,n_3]$ with the induction datum $(t_j^{(i)})_{ij}$ satisfying  
$t^{(4)}_{1}=t^{(4)}_{2}=t^{(5)}_{2}=t^{(5)}_{3}=t^{(6)}_{1}=t^{(6)}_{3}=1/2$, 
$t^{(7)}_{1}=t^{(7)}_{2}=t^{(7)}_{3}=1/3$, 
$t^{(8)}_{1}=t^{(8)}_{7}=t^{(9)}_{2}=t^{(9)}_{7}=t^{(10)}_{3}=t^{(10)}_{7}=1/2$, 
and $t^{(i)}_j=0$ otherwise. 
We now use Lemma \ref{lem:abstraction}. Let $(s_i)_i$ be as in the lemma; 
thus in particular $\{s_1,s_2,s_3\}=\{\pm\}$.  We have to verify that
every $n_j$ in the interior of the triangle has a constant sign path
to the boundary.  Indeed, $n_7$ is connected to the $n_j$ for
$j=1,2,3$ if ${s}_j={s}_7$ and $n_8$ is also
connected to a boundary point as otherwise
${s}_1={s}_6={s}_3\neq {s}_8$, and
therefore
${s}_8={s}_7={s}_2={s}_9$, a
contradiction which establishes the claim by symmetry.

\end{enumerate}
\bibliographystyle{amsalpha}

\end{document}